\documentclass[letterpaper]{amsart}

\usepackage[utf8]{inputenc}
\usepackage[T1]{fontenc}
\usepackage{lmodern}
\usepackage{amssymb}
\usepackage{mathtools}

\usepackage{tikz-cd}
\usepackage[pdfusetitle]{hyperref}
\def\doi#1{\href{http://dx.doi.org/#1}{doi:#1}}

\mathtoolsset{mathic=true}
\let\intertext\shortintertext

\theoremstyle{plain}
\newtheorem{theorem}{Theorem}[section]
\newtheorem{proposition}[theorem]{Proposition}
\newtheorem{lemma}[theorem]{Lemma}
\newtheorem{corollary}[theorem]{Corollary}

\theoremstyle{definition}

\newtheorem{example}[theorem]{Example}
\newtheorem{remark}[theorem]{Remark}

\newtheorem{assumption}[theorem]{Assumption}

\theoremstyle{remark}

\numberwithin{equation}{section}

\DeclareMathAlphabet\mathbfit{OML}{cmm}{b}{it}

\def\Z{\mathbb Z}
\def\Q{\mathbb Q}

\def\CP{\mathbb{CP}}

\DeclareMathOperator{\Hom}{Hom}
\DeclareMathOperator{\Ext}{Ext}
\DeclareMathOperator{\Tor}{Tor}

\DeclareMathOperator{\im}{im}

\let\MFsm\setminus
\def\setminus{\mathbin{\mkern-2mu\MFsm\mkern-2mu}}

\def\HH{H}
\def\OO{\mathcal O}
\def\II{\mathcal I}
\def\DD{\mathcal D}
\def\III{\mathbb I}

\def\kk{\Bbbk}

\def\dirlim{\mathop{\underrightarrow\lim}}

\DeclareMathOperator{\depth}{depth}
\DeclareMathOperator{\ann}{ann}

\def\pair#1{{\langle#1\rangle}}

\def\HT{H_{T}}

\def\CT{C_{T}}

\def\veeHT{{\check H}_{T}}
\def\veeCT{{\check C}_{T}}
\def\veehCT{{\check C}^{T}}

\def\AB{AB}

\def\barAB{{\smash{\overline{AB\mathstrut}}\mathstrut}}

\def\m{\mathfrak m}

\def\hCT{C^{T\!}}

\def\hHT{H^{T\!}}

\def\hathHT{\hat{H}^{T\!}}
\def\hatCT{\hat{C}_{T}}
\def\hatHT{\hat{H}_{T}}

\def\dd{d}

\def\eg{\emph{e.\,g.}}
\def\ie{\emph{i.\,e.}}
\def\cf{\emph{cf.}}
\def\Cf{\emph{Cf.}}

\begin{document}

\title[Equivariant cohomology, syzygies and orbit structure]{Equivariant cohomology, syzygies \\ and orbit structure}
\author{Christopher Allday}
\address{Department of Mathematics, University of Hawaii,
  2565~McCarthy Mall, Honolulu, HI~96822, U.S.A.}
\email{chris@math.hawaii.edu}
\author{Matthias Franz}
\address{Department of Mathematics, University of Western Ontario,
      London, Ont.\ N6A\;5B7, Canada}
\email{mfranz@uwo.ca}
\author{Volker Puppe}
\address{Fachbereich Mathematik, Universität Konstanz, 78457 Konstanz, Germany}
\email{volker.puppe@uni-konstanz.de}
\thanks{M.\,F.\ was partially supported by an NSERC Discovery Grant.}

\hypersetup{pdfauthor=\authors}


\begin{abstract}
  Let \(X\) be a ``nice'' space with an action of a torus~\(T\).
  We consider the Atiyah--Bredon sequence of equivariant cohomology modules
  arising from the filtration of~\(X\) by orbit dimension.
  We show that a front piece of this 
  sequence is
  exact if and only if the \(H^{*}(BT)\)-module~%
  \(\HT^{*}(X)\) is a certain syzygy.
  Moreover, we express
  the cohomology of that sequence
  as an \(\Ext\)~module involving a suitably defined equivariant homology of~\(X\).

  One consequence is that the GKM~method for computing equivariant cohomology
  applies to a Poincaré duality space 
  if and only if 
  the equivariant Poincaré pairing is perfect.
\end{abstract}

\maketitle


\section{Introduction}

Consider an action of the torus~\(T\cong(S^{1})^{r}\) on a space~\(X\)
satisfying some mild conditions (stated in
Assumption~\ref{ass:orbit-filtration}).
Let \(X_{i}\) be the \(T\)-equivariant \(i\)-skeleton of~\(X\), \ie, the union of all orbits
of dimension at most~\(i\).
By a result of Chang--Skjelbred~ \cite[Prop.~2.4]{ChangSkjelbred:1974},
the sequence 
\begin{equation}
  \label{eq:chang-skjelbred}
  \let\longrightarrow\rightarrow
  0
  \longrightarrow \HT^{*}(X)
  \longrightarrow \HT^{*}(X_0)
  \longrightarrow \HT^{*+1}(X_1, X_0)
\end{equation}
is exact if the equivariant cohomology~\(\HT^{*}(X)\) with rational coefficients
is free over the polynomial ring~\(R=H^{*}(BT)\).
Roughly at the same time,
Atiyah~\cite{Atiyah:1974} (in the context of equivariant \(K\)-theory)
and Bredon~\cite{Bredon:1974} (using Atiyah's method) proved
the stronger result that under the same hypothesis the following longer sequence is exact:
\begin{equation}\label{eq:atiyah-bredon}
  \let\longrightarrow\rightarrow
  0
  \longrightarrow \HT^{*}(X)
  \longrightarrow \HT^{*}(X_0)
  \longrightarrow \HT^{*+1}(X_1, X_0)
  \longrightarrow \cdots
  \longrightarrow \HT^{*+r}(X_r, X_{r-1})
  \longrightarrow 0.
\end{equation}
In recent years, this ``Atiyah--Bredon sequence'' has been studied by Franz--Puppe \cite{FranzPuppe:2007},~\cite{FranzPuppe:2011}
and Goertsches--Töben~\cite{GoertschesToeben:2010}. Moreover, it is implicit
in papers of De\,Con\-cini--Procesi--Vergne on transversally elliptic operators, splines
and the infinitesimal index~\cite{DeConciniProcesiVergne:2010}, \cite{DeConciniProcesiVergne:2011}.
It is also related to Schenck's work on splines and
equivariant Chow groups of toric varieties~\cite{Schenck:1997},~\cite{Schenck:2012}
and to the generalization of intersection cohomology for toric varieties
studied by Barthel--Brasselet--Fieseler--Kaup \cite{BarthelBrasseletFieselerKaup:2002}.

The assumption that \(\HT^{*}(X)\) be a free \(R\)-module
is known to hold for large classes of spaces,
including compact Hamiltonian \(T\)-manifolds
and rationally smooth, projective 
complex algebraic varieties with an algebraic action of the complexification of~\(T\),
\cf~\cite[Thm.~14.1]{GoreskyKottwitzMacPherson:1998}.
In all these cases
the ``Chang--Skjelbred sequence''~\eqref{eq:chang-skjelbred} provides a powerful way to compute \(\HT^{*}(X)\),
including the cup~product, out of data
related only to the fixed points and the one-dimensional orbits.
In the important special case where \(X_{1}\) is a finite union of \(2\)-spheres,
glued together at their poles,
this is often referred to as the ``GKM method'',
following work of Goresky--Kottwitz--MacPherson~\cite[Thm.~7.2]{GoreskyKottwitzMacPherson:1998}.
It should be noted that one only needs exactness of a very small part of the Atiyah--Bredon sequence
in order to apply this method.
This suggests that the sequence~\eqref{eq:chang-skjelbred}
might be exact under much weaker assumptions than the freeness of~\(\HT^{*}(X)\).

In this paper we address the following questions:
\begin{enumerate}
\item \label{qq1}
  Under which condition is the Chang--Skjelbred sequence exact? In particular:
  For which \(T\)-spaces does the GKM~method work?
\item If the Atiyah--Bredon sequence is not exact, what is the meaning
  of its cohomology?
\end{enumerate}

In fact, we look at a more general question than~\eqref{qq1}:
Under which condition is the Atiyah--Bredon sequence exact
from the left up to (and including) the \(i\)-th position,
\ie, up to the term~\(\HT^{*}(X_{i},X_{i-1})\)?
(We refer to~\(\HT^{*}(X)\) as position~\(i=-1\).)

To answer this question, we need a notion from commutative algebra.
A finitely generated \(R\)-module~\(M\) is called a \emph{\(j\)-th syzygy}
if there is an exact sequence
\begin{equation}
  0\to M\to F^{1}\to \dots \to F^{j}
\end{equation}
with finitely generated free \(R\)-modules~\(F^{1}\),~\ldots,~\(F^{j}\).
The first syzygies are exactly the torsion-free \(R\)-modules,
and the \(j\)-th syzygies with~\(j\ge r\) are the free modules, \cf~Section~\ref{sec:Torsion-freeness}.
The following, which is  part of Theorem \ref{thm:conditions-partial-exactness}, therefore
implies Atiyah--Bredon's result as well as its converse.

\begin{theorem}
  \label{thm:intro-partial-exactness}
  Let \(j\ge0\). Then
  the Atiyah--Bredon sequence~\eqref{eq:atiyah-bredon}
  is exact at all positions~\(i\le j-2\)
  if and only if
  \(\HT^{*}(X)\) is a \(j\)-th syzygy.  
\end{theorem}


To address the second question, we consider
a suitably defined \emph{equivariant homology}~\(\hHT_{*}(X)\) of~\(X\).
We stress that this is \emph{not} the homology of the Borel construction~\(X_{T}\), see Section~\ref{sec:equiv-homology}.
Equivariant Poincaré duality holds in the sense that
for a rational Poincaré duality space~\(X\) of formal dimension~\(n\)
capping with the equivariant fundamental class
gives an isomorphism~\(\HT^{*}(X)\to \hHT_{*}(X)\) of degree~\(-n\).
Moreover,
let \(H^{*}(\AB^{*}(X))\) be the cohomology
of the Atiyah--Bredon sequence~\eqref{eq:atiyah-bredon},
considered as a complex of \(R\)-modules and with the term~\(\HT^{*}(X)\)
omitted. Our main result,
Theorem~\ref{thm:exthab-ss}, implies 
that \(H^{*}(\AB^{*}(X))\) is completely determined by~\(\hHT_{*}(X)\):

\begin{theorem}
  \label{thm:intro-main}
  Let \(X\) be a \(T\)-space.
  For any~\(j\ge0\) there is an isomorphism of \(R\)-modules
  \begin{equation*}
    H^{j}(\AB^{*}(X)) = \Ext_{R}^{j}(\hHT_{*}(X),R).
  \end{equation*}
  In particular,
  if \(X\) is a rational Poincaré duality space of formal dimension~\(n\), then for any~\(j\ge0\)
  there is an isomorphism of \(R\)-modules of degree~\(-n\)
  \begin{equation*}
    H^{j}(\AB^{*}(X)) = \Ext_{R}^{j}(\HT^{*}(X),R).
  \end{equation*}
\end{theorem}

The Atiyah--Bredon sequence is the \(E_{1}\)~page of the spectral sequence
induced by the equivariant skeletons~\(X_{i}\) and converging to~\(\HT^{*}(X)\).
Similarly, \(\Ext_{R}^{*}(\hHT_{*}(X),R)\) is the \(E_{2}\)~page
of a universal coefficient spectral sequence computing \(\HT^{*}(X)\) out of
the equivariant chains on~\(X\),
see Section~\ref{sec:universal-coeff}. In Theorem~\ref{thm:exthab-ss} we
actually prove
that these two spectral sequences are isomorphic
from the \(E_{2}\)~page on.

Let \(X\) be a rational Poincaré duality space.
Since the equivariant coefficient ring~\(R\) is not a field (unless \(r=0\)),
the 
isomorphism between \(\HT^{*}(X)\) and \(\hHT_{*}(X)\)
does not imply that the corresponding equivariant Poincaré pairing
\begin{equation}
    \label{eq:poincare-cup-equiv-intro}
  \HT^{*}(X) \otimes \HT^{*}(X) \to R
\end{equation}
is non-degenerate, let alone perfect.
For instance, one has~\(\HT^{*}(X)=\Q\) for~\(X=T\),
so that the map~\eqref{eq:poincare-cup-equiv-intro}
is trivial in this case.

Recall that
a (graded symmetric) \(R\)-bilinear pairing~%
\( 
  M\times M\to R
\) 
is called \emph{perfect} if it induces an isomorphism \(M\to\Hom_{R}(M,R)\).
Moreover,
an \(R\)-module~\(M\) is called \emph{reflexive}
if the canonical map
\begin{equation}
  M\to\Hom_{R}(\Hom_{R}(M,R),R)
\end{equation}
is an isomorphism.
Because reflexive \(R\)-modules are exactly the second syzygies,
our next result is an immediate consequence of
Theorems \ref{thm:intro-partial-exactness}~and~\ref{thm:intro-main}.
It essentially answers an open point raised by
Guillemin--Ginzburg--Karshon, 
see Remark~\ref{rem:ggk}.

\begin{corollary}
  Let \(X\) be a rational Poincaré duality space. Then the following are equivalent:
  \begin{enumerate}
  \item The Chang--Skjelbred sequence~\eqref{eq:chang-skjelbred} is exact.
  \item The \(R\)-module~\(\HT^{*}(X)\) is reflexive.
  \item The equivariant Poincaré pairing~\eqref{eq:poincare-cup-equiv-intro}
    is perfect.
  \end{enumerate}
\end{corollary}


For any~\(j\ge-1\) there are \(T\)-spaces 
such that the sequence~\eqref{eq:atiyah-bredon} is exact
at all positions~\(i\le j\),
but not further, see Section~\ref{sec:toric}.
The situation changes if one restricts to rational Poincaré duality spaces:
Allday~\cite{Allday:1985} has shown that
in this case torsion-freeness of~\(\HT^{*}(X)\) 
implies freeness if~\(r=2\), \ie, if \(T\cong S^{1}\times S^{1}\).
There are counterexamples for~\(r\ge3\) due to Franz and Puppe, \cf~Section~\ref{sec:mutants}.
The correct generalization of Allday's result 
is as follows. In the light of Theorem~\ref{thm:intro-partial-exactness}, our result
says roughly that
if the first half of the Atiyah--Bredon sequence is exact, then so is the rest:

\begin{corollary}
  Let \(X\) be a rational Poincaré duality space.
  If \(\HT^{*}(X)\) is a syzygy of order~\(\ge r/2\), then it is free over~\(R\).
\end{corollary}

\medbreak

The paper is organized as follows:
In Section~\ref{sec:algebra} we collect some results from commutative algebra
that we will need later on. In particular, we discuss
Cohen--Macaulay modules, syzygies and the Koszul resolution.
In Section~\ref{sec:equiv-cohomology} we review the singular Cartan model
for \(T\)-equivariant cohomology and introduce equivariant homology.
As in the non-equivariant theory,
homology and cohomology are related via a universal coefficient theorem
and via Poincaré duality.
In Section~\ref{sec:main-result} we prove our main result,
the spectral sequence version of Theorem~\ref{thm:intro-main}.
Section~\ref{sec:applications} contains consequences of the main result,
in particular for the partial exactness of the Atiyah--Bredon sequence
and for Poincaré duality spaces.
We conclude with two examples in Section~\ref{sec:examples}.

\section{Algebraic background}
\label{sec:algebra}

\subsection{Standing assumptions}

Unless specified otherwise, we work over a ground field~\(\kk \) of arbitrary characteristic,
and all tensor products are taken over~\(\kk \).
The letter~\(R\) denotes
a polynomial ring in \(r\)~indeterminates of degree~\(\dd\ge1\) with coefficients in~\(\kk \),
and \(\m\lhd R\) its maximal graded ideal.
All \(R\)-modules are assumed to be graded.
We consider \(\kk \) as an \(R\)-module (concentrated in degree~\(0\)) via the canonical augmentation.
For an \(R\)-module~\(M\) and an~\(l\in\Z\)
the notation~\(M[l]\) denotes a degree shift by~\(l\), so that the degree~\(l'\)~piece of~\(M[l]\)
is the degree~\(l'-l\)~piece of~\(M\).
We write ``\(\subset\)'' for (not necessarily proper) inclusion of sets.

\smallskip

In the following subsections we review some notions from commutative algebra.
Apart from the references given below,
the reader might also
find the summary of results 
in~\cite[App.~A]{AlldayPuppe:1993} helpful.
They were compiled with applications in equivariant cohomology in mind.

All \(R\)-modules are assumed to be finitely generated for the rest of this section.

\subsection{Cohen--Macaulay modules}

Let \(M\) be an \(R\)-module.
A sequence \(a_{1}\),~\ldots,~\(a_{j}\) in~\(R\) is \emph{\(M\)-regular}
if \(a_{i}\) is not a zero-divisor in~\(M/(a_{1},\dots,a_{i-1})M\)
for~\(1\le i\le j\).
If an \(M\)-regular sequence of length~\(r\) exists, then \(M\) is free over~\(R\).

Assume \(M\ne 0\). The depth of~\(M\)
is the common length of all maximal \(M\)-regular sequences.
One always has
\begin{equation}
  \depth M \le \dim M,
\end{equation}
where \(\dim M\) is the Krull dimension of the ring~\(R/{\ann M}\).
If equality holds, one calls \(M\) \emph{Cohen--Macaulay}.
We will often make use of the following characterization,
\cf~the proof of~\cite[Prop.~A1.16]{Eisenbud:2005}.

\begin{proposition}
  \label{thm:dim-depth-ext}
  Let \(M\) be a non-zero \(R\)-module.
  \begin{enumerate} 
    \item \(\dim M\) is the largest integer~\(i\) such that \(\Ext_{R}^{r-i}(M,R)\ne0\).
    \item \(\depth M\) is the smallest integer~\(i\) such that \(\Ext_{R}^{r-i}(M,R)\ne0\).
    \item \(M\) is Cohen--Macaulay of dimension~\(j\) if and only if \(\Ext_{R}^{r-i}(M,R)=0\) for all~\(i\ne j\).
  \end{enumerate}
\end{proposition}

The following well-known property of Cohen--Macaulay modules will be crucial for us,
\cf~
\cite[Lemma~7.5]{Atiyah:1974} or~\cite[Cor.~A.6.16]{AlldayPuppe:1993}.

\begin{lemma}
  \label{thm:CM-map-0}
  Let \(M\) be a Cohen--Macaulay \(R\)-module of dimension~\(j\).
  Then any non-zero submodule of~\(M\) has dimension~\(j\).
  Equivalently, any map~\(N\to M\) of \(R\)-modules
  is trivial if \(\dim N<j\).
\end{lemma}

\subsection{Torsion-freeness}
\label{sec:Torsion-freeness}

Any free \(R\)-module is torsion-free, but of course the converse is false for~\(r>1\).
We review a useful way to interpolate between these two notions.
A good reference for this material is \cite[Sec.~16E]{BrunsVetter:1988}.

Reflexive \(R\)-modules and syzygies have been defined in the introduction.
For convenience, we call any \(R\)-module a zeroth syzygy.
By Hilbert's Syzygy Theorem,  an \(R\)-module~\(M\) is an \(r\)-th syzygy if and only if
it is free over~\(R\). This holds for not finitely-generated~\(M\) as well.

\begin{proposition}
  \label{thm:torsionfree}
  The following are equivalent for any \(R\)-module~\(M\) and any \(j\ge 1\):
  \begin{enumerate}
  \item \label{tf1} \(M\) is a \(j\)-th syzygy.
  \item \label{tf2} Every \(R\)-regular sequence of length at most~\(j\) is \(M\)-regular.
  \item \label{tf3}
    One of the following conditions holds, depending on~\(j\):
    \begin{enumerate}
    \item[\(j=1\):] \(M\) is torsion-free.
    \item[\(j=2\):] \(M\) is reflexive.
    \item[\(j\ge3\):] \(M\) is reflexive and \(\Ext_{R}^{i}(\Hom_{R}(M,R),R)=0\) for all~\(1\le i\le j-2\).
    \end{enumerate}
  \end{enumerate}
\end{proposition}

\subsection{The Koszul resolution}
\label{sec:koszul-resolution}

The easiest way
to obtain syzygies over a polynomial ring is to use the Koszul resolution
\begin{equation}
  \label{eq:koszul-resolution}
  0 \longrightarrow
  R[\dd r] \stackrel{\delta_{r}}\longrightarrow
  R[\dd(r-1)]^{\binom{r}{r-1}} \longrightarrow
  \cdots \longrightarrow
  R[\dd]^{\binom{r}{1}} \stackrel{\delta_{1}}\longrightarrow
  R \stackrel{\delta_{0}}\longrightarrow \kk  \longrightarrow 0 \; ;
\end{equation}
indeed, the image of~\(\delta_{j}\) is a \(j\)-th syzygy by definition.
We define
\begin{equation}
  K_{j} = \im \delta_{j}[-\dd j] = \ker \delta_{j-1}[-\dd j]
\end{equation}
The degrees shifts ensure that \(K_{j}\) is generated in degree~\(0\).
For example, \(K_{1}=\m[-\dd]\) and \(K_{r}=R\).
(Recall that the indeterminates have degree~\(\dd\).)
We set \(K_{0}=\kk \) and \(K_{r+1}=0\) for convenience.

\goodbreak

In Section~\ref{sec:examples}
we will need the following property of the Koszul syzygies.

\begin{lemma}
  \label{thm:Ext-Koszul}
  For~\(1\le j\le r\),
  \begin{equation*}
    \Ext_{R}^{i}(K_{r-j},R) \cong
    \begin{cases}
      K_{j+1}[\dd] & \text{if \(i=0\),} \\
      \kk [-\dd j] & \text{if \(i=j\),} \\
      0 & \text{otherwise.}
    \end{cases}
  \end{equation*}
  Moreover, any short exact sequence of \(R\)-modules
  \begin{equation*}
    0 \longrightarrow R[l] \longrightarrow M \longrightarrow K_{r-j}[l'] \longrightarrow 0
  \end{equation*}
  splits if \(j\ne1\) or \(l-l'\ne \dd\).
\end{lemma}

\begin{proof}
  The first claim is an easy calculation based on the self-duality of
  the Koszul resolution. The second part follows from the first and the fact
  that (graded) extensions of the form above are classified
  by the degree~\(0\) part of
  \begin{equation}
    \Ext_{R}^{1}(K_{r-j}[l'],R[l])=\Ext_{R}^{1}(K_{r-j},R)\bigl[l-l'\bigr].
    \qedhere
  \end{equation}
\end{proof}

The observation that extensions of \(R\)-modules
\begin{equation}
  \label{eq:extension-LMN}
  0 \longrightarrow L \longrightarrow M \longrightarrow N \longrightarrow 0
\end{equation}
(with maps of degree~\(0\)) 
are classified by the degree~\(0\) part of~\(\Ext_{R}^{1}(N,L)\)
is certainly not new, but we could not locate it in the literature.
It can be proven in the same way as in the ungraded case:
Any free resolution~\(F_{2}\to F_{1}\to F_{0}\to N\) gives rives
to a commutative diagram
\begin{equation}
  \begin{tikzcd}[row sep=large, column sep=large]
    F_{2} \dar{f_{2}} \rar & F_{1} \dar{f_{1}} \rar & F_{0} \dar{f_{0}} \rar & N \dar{=} \rar & 0 \\
    0 \rar & L \rar & M \rar & N \rar & 0 \mathrlap{.}
  \end{tikzcd}
\end{equation}
The map~\(f_{1}\)
determines a degree~\(0\) class in~\(\Ext_{R}^{1}(N,L)\). This class is independent
of the choices made, and is zero if and only if the extension~\eqref{eq:extension-LMN}
splits.

\section{Equivariant homology and cohomology}
\label{sec:equiv-cohomology}

\subsection{Standing assumptions}
\label{sec:assumptions}

\(C_{*}(-)\) and \(C^{*}(-)\) denote normalized singular chains and cochains
with coefficients in the field~\(\kk \), and \(H_{*}(-)\)~and~\(H^{*}(-)\)
singular (co)ho\-mol\-ogy.
We adopt a cohomological grading, so that the homology
of a space lies in non-positive degrees; an element~\(c\in H_{i}(X)\)
has cohomological degree~\(-i\).

Throughout, \(T\cong(S^{1})^{r}\) denotes a torus of rank~\(r\).
The cohomology ring
\( 
  R=H^{*}(BT) 
\) 
of its classifying space
is a polynomial algebra on \(r\)~generators of degree~\(2\).

All \(T\)-spaces 
are assumed to be Hausdorff, second-countable, locally compact and locally contractible.
Important examples are topological manifolds, complex algebraic varieties,
and countable, locally finite CW~complexes,
each with a continuous \(T\)-action.

It follows from our assumptions that every subset~\(Y\subset X\) is paracompact, hence singular cohomology
and Alexander--Spanier cohomology are naturally isomorphic for all \(T\)-pairs~\((X,Y)\)
such that \(Y\) is locally contractible.
The latter will be a standing assumption on all \(T\)-pairs we consider;
it holds automatically if \(Y\) is open in~\(X\).

\subsection{The singular Cartan model}
\label{sec:singular-Cartan-model}

It will be convenient to use the ``singular Cartan model''
\cite[Sec.~7.3]{FelixHalperinThomas:1995}
(see also \cite[Lemma~5.1]{Jones:1987}~and~\cite[Sec.~5.1]{Franz:2003}).
We recall the construction. 
For smooth manifolds and real coefficients,
the exposition could be simplified by substituting
the usual Cartan model~\(\Omega_{T}^{*}(X)=\Omega^{*}(X)\otimes R\)
for the singular Cartan model~\(\CT^{*}(X)\).
In particular, this would avoid some technical difficulties addressed
in Remark~\ref{rem:finiteness}.

The normalized singular chain complex~\(C_{*}(T)\) is a (graded) commutative dg~bialgebra
via the Pontryagin product and the Alexander--Whitney diagonal~\(\Delta\).
Moreover \(C_{*}(A,B)\)~and \(C^{*}(A,B)\) are
naturally dg~modules over~\(C_{*}(T)\) for any \(T\)-pair~\((A,B)\) in~\(X\). 

The \(\kk \)-vector spaces \(H_{1}(T)\)~and~\(H^{2}(BT)\)
are canonically dual to each other
by the transgression homomorphism~\(H^{1}(T)\to H^{2}(BT)\).
We choose dual bases \((x_{1},\dots,x_{r})\) of~\(H_{1}(T)\) and
\((t_{1},\dots,t_{r})\) of~\(H^{2}(BT)\)
as well as representatives~\(a_{i}\in C_{1}(T)\) of the homology classes~\(x_{i}\).
We require the~\(a_{i}\) to be primitive, \ie,
\(\Delta a_{i}=a_{i}\otimes1+1\otimes a_{i}\) for all~\(i\).
For example, if the basis~\((x_{i})\) is induced by a basis of~\(\pi_{1}(T,1)\),
then the~\(a_{i}\)'s can be representative loops.

For a \(T\)-pair~\((A,B)\) in~\(X\), consider the free \(R\)-module
\begin{align}
  \label{eq:definition-CT}
  \CT^{*}(A,B) &= C^{*}(A,B)\otimes R.
  \intertext{The assignments}
  \label{eq:definition-d-CT}
  d(\gamma\otimes f) &=
  d\gamma\otimes f + \sum_{i=1}^{r}a_{i}\cdot\gamma\otimes t_{i}f \\
  (\gamma\otimes f)\cup(\gamma'\otimes f') &= \gamma\cup\gamma'\otimes f f'
\end{align}
turn \(\CT^{*}(A,B)\) into a dg~\(R\)-algebra, that is,
a dg~algebra and dg~\(R\)-module such that the product is \(R\)-bilinear.
(Here one uses that \(C_{*}(T)\) is (graded) commutative
and that the representatives~\(a_{i}\) are primitive.)

We define the \(T\)-equivariant cohomology of~\((A,B)\) by
\begin{equation}
  \HT^{*}(A,B) = H^{*}(\CT^{*}(A,B))
\end{equation}
and, following \cite{Franz:2003}, call \eqref{eq:definition-CT}
the \emph{singular Cartan model} of~\((A,B)\).
The \(R\)-algebra \(\HT^{*}(A,B)\) is naturally isomorphic
to~\(H^{*}(A_{T},B_{T})\)
\cite[Thm.~7.5]{FelixHalperinThomas:1995},~\cite[Thm.~5.1]{Franz:2003},
where \(X_{T}=(ET\times X)/T\) denotes the Borel construction
(or homotopy quotient) of~\(X\).
In particular, \(\HT^{*}(A,B)\) does not depend on the choices made above.

Filtering \eqref{eq:definition-CT} by \(R\)-degree leads to a
first quadrant spectral sequence with
\begin{equation}
  \label{eq:ss-HT-E2}
  E_{1} = E_{2} = H^{*}(A,B)\otimes R  \;\Rightarrow\;  \HT^{*}(A,B).
\end{equation}
This spectral sequence
is isomorphic to the Serre spectral sequence for the fibration~\(X\to X_{T}\to BT\)
from the \(E_{2}\)~page on, \cf~the proof of~\cite[Thm.~4.7]{Franz:2003}.

\begin{remark}
  \label{rem:consequences-Serre}
  It follows from the spectral sequence~\eqref{eq:ss-HT-E2}
  (or the minimal Hirsch--Brown model mentioned in Remark~\ref{rem:finiteness})
  that the \(R\)-module~\(\HT^{*}(A,B)\)
  is free if the restriction map~\(\HT^{*}(A,B)\to H^{*}(A,B)\)
  is surjective (Leray--Hirsch) and that it is finitely generated
  if \(H^{*}(A,B)\) is finite-dimensional over~\(\kk\).
  The converses to both statements are true as well; to see this, one can for instance
  invoke 
  an Eilenberg--Moore theorem, \cf~\cite[Ch.~7]{McCleary:2001}.
\end{remark}

\begin{assumption}
  \label{ass:H-finite}
  We additionally assume from now on that \(H^{*}(A,B)\) is finite-dimensional
  for all \(T\)-pairs~\((A,B)\) we consider. By Remark~\ref{rem:consequences-Serre},
  this is equivalent to
  \(\HT^{*}(A,B)\) being finitely generated over~\(R\).
\end{assumption}

\subsection{Equivariant homology}
\label{sec:equiv-homology}

We define the equivariant chain complex~\(\hCT_{*}(A,B)\) of the \(T\)-pair~\((A,B)\)
to be the (graded) \(R\)-dual of the singular Cartan model,
\begin{equation}
  \hCT_{*}(A,B) = \Hom_{R}(\CT^{*}(A,B),R).
\end{equation}
(Recall that a map has degree~\(m\in\Z\) if it shifts degrees by~\(m\).)

The \emph{equivariant homology} of~\((A,B)\)
is defined as~\(\hHT_{*}(A,B)=H_{*}(\hCT_{*}(A,B))\).
Note that \(\hHT_{*}(X)\) is \emph{not} the homology of the Borel construction~\(X_{T}\) in general.
For example, for a point~\(X=*\), \(\hHT_{*}(X)=R\) is free over~\(R\)
whereas \(H_{*}(X_{T})=\Hom_{\kk }(\HT^{*}(X),\kk )\) is torsion.
We will see in Remark~\ref{rem:finiteness}
that \(\hHT_{*}(A,B)\) is bounded below under our assumptions on spaces.

It turns out that \(\hHT_{*}(X)\) is a suitable equivariant homology
in the sense that it enjoys many desirable properties.
For example, equivariant homology is related to equivariant cohomology
via a universal coefficient theorem over~\(R\) (Proposition~\ref{thm:uct})
and, in case of a Poincaré duality space,
also via equivariant Poincaré duality (Proposition~\ref{thm:equiv-PD}).
It is therefore not surprising that other people have considered this or
similar constructions before.
The earliest we are aware of are Jones~\cite[\S 5]{Jones:1987}
(for \(T=S^{1}\)), Brylinski~\cite{Brylinski:1992} (for intersection homology)
and Edidin--Graham~\cite[Sec.~2.8]{EdidinGraham:1998} 
(for algebraic varieties and homology with closed supports).
Equivariant homology is also implicit in~\cite[p.~353]{AlldayPuppe:1993}.

\begin{remark}
  \label{rem:finiteness}
  The dg~\(R\)-module~\(\hCT_{*}(A,B)\) is not bounded below, which will
  make convergence of spectral sequences a delicate issue.
  (In contrast, the \(R\)-dual of the usual Cartan model is bounded below
  as \(\Omega^{*}(X)\) is bounded above.) It is
  therefore useful to observe that both \(\CT^{*}(A,B)\)~and~\(\hCT_{*}(A,B)\)
  are \(R\)-homotopy equivalent to dg~\(R\)-modules
  which are free as \(R\)-modules and bounded below.  
  It follows that \(\hHT_{*}(A,B)\) is bounded below as well.
  
  For~\(\CT^{*}(A,B)\), the ``minimal Hirsch--Brown model''~\(H^{*}(A,B)\otimes R\)
  (with a twisted differential) is one such replacement \cite[Cor.~B.2.4]{AlldayPuppe:1993}.
  Since \(\Hom_{R}(-,R)\) preserves \(R\)-homotopies, the claim for~\(\hCT_{*}(A,B)\) follows.

  Tensoring the \(R\)-homotopy equivalences between \(\CT^{*}(A,B)\) and
  a finitely generated free replacement~\(N\otimes R\) with~\(\kk \) over~\(R\)
  yields homotopy equivalences between \(C^{*}(A,B)\) and~\(N\).
  Hence, if one filters both dg~\(R\)-modules by \(R\)-degree,
  then the maps of spectral sequences induced by the \(R\)-homotopy equivalences
  become isomorphisms from the \(E_{1}\)~page one. The same applies to~\(\hCT_{*}(A,B)\).
  Therefore, in any argument involving a
  comparison of spectral sequences obtained as above,
  one need not worry about convergence issues for~\(\hCT_{*}(A,B)\)
  as one could always replace this complex by one which is bounded below,
  without affecting the pages from \(E_{1}\)~on.

  In particular, there is a convergent spectral sequence
  \begin{equation}
    \label{eq:ss-hHT-E2}
    E_{1} = E_{2} = H_{*}(A,B)\otimes R  \;\Rightarrow\;  \hHT_{*}(A,B)
  \end{equation}
  analogous to~\eqref{eq:ss-HT-E2}.
  Hence any equivariant map of \(T\)-pairs~\((A,B)\to(A',B')\)
  which is a non-equivariant quasi-isomorphism
  induces an isomorphism not only in equivariant cohomology
  (by virtue of the Serre spectral sequence~\eqref{eq:ss-HT-E2}),
  but also in equivariant homology.
  (In fact, it induces an isomorphism between the minimal Hirsch--Brown models.)
\end{remark}

We remarked in Section~\ref{sec:singular-Cartan-model}
that \(\HT^{*}(X)\) does not depend on the choice
of representatives~\(a_{i}\in C_{1}(T)\), hence
neither on the chosen basis~\(x_{1}\),~\ldots,~\(x_{r}\in H_{1}(T)\).
Although we will not need it in the sequel,
we now prove the analogous statement for~\(\hHT_{*}(X)\)
for the sake of completeness;
an alternative proof for~\(\HT^{*}(X)\) will be given along the way.

\begin{proposition}
  \label{thm:hHT-independent-basis}
  Equivariant (co)homology does not depend on the choice of representatives~\(a_{i}\in C_{1}(T)\).
  More precisely:
  Let \(\hat a_{1}\),~\dots,~\(\hat a_{r}\) be another set of representatives,
  and denote by~\(\hatHT^{*}(A,B)\) and \(\hathHT_{*}(A,B)\) the equivariant (co)homology defined via them.
  Then \(\HT^{*}(A,B)\) and \(\hatHT^{*}(A,B)\) are naturally isomorphic as \(R\)\nobreakdash-modules,
  for all \(T\)-pairs~\((A,B)\). 
  The same holds for \(\hHT_{*}(A,B)\) and \(\hathHT_{*}(A,B)\).
\end{proposition}

\begin{proof}
  We have \(\hat a_{i}-a_{i}= d(b_{i})\) for some chains~\(b_{i}\in C_{2}(T)\),
  \(i=1,\dots,r\). The map
  \begin{equation}
    \CT^{*}(A,B) \to \hatCT^{*}(A,B),
    \quad
    \gamma\otimes f \mapsto
    \gamma\otimes f - \sum_{i=1}^{r} b_{i}\cdot \gamma\otimes t_{i}f
  \end{equation}
  is a morphism of dg~\(R\)-modules, and the induced map between the
  \(E_{1}\)~pages~\eqref{eq:ss-HT-E2} is the identity.
  Hence \(\HT^{*}(A,B)\)~and~\(\hatHT^{*}(A,B)\) are naturally isomorphic.
  The claim for~\(\hHT_{*}(A,B)\) follows by dualizing.
\end{proof}

\subsection{Universal coefficient theorem}
\label{sec:universal-coeff}

In the case of an ungraded coefficient ring, universal coefficient theorems
are standard results in homological algebra. We need the following variant.

\begin{proposition} 
  \label{thm:uct}
  Let \((A,B)\) be a \(T\)-pair. 
  Then there are spectral sequences,
  natural in~\((A,B)\),
  \begin{align*}
    E_{2}^{p} &= \Ext_{R}^{p}(\HT^{*}(A,B),R) \;\Rightarrow\; \hHT_{*}(A,B), \\
    E_{2}^{p} &= \Ext_{R}^{p}(\hHT_{*}(A,B),R) \;\Rightarrow\; \HT^{*}(A,B).
  \end{align*}
\end{proposition}

\begin{proof}
  Since \(\hCT_{*}(A,B)\) is defined as the \(R\)-dual of~\(\CT^{*}(A,B)\),
  the existence of the first spectral sequence follows
  from the following claim:
  Let \(M\) be a dg~\(R\)-module which is free over~\(R\)
  if one forgets the differential.
  Then there is a spectral sequence converging
  to~\(H^{*}(\Hom_{R}(M,R))\) with \(E_{2}\)~page~\(\Ext_{R}(H^{*}(M),R)\).

  To prove this claim,
  we take a (graded) injective resolution
  \begin{equation*}
    0\to R\to I^{0}\to\dots\to I^{r}\to 0
  \end{equation*}
  with total complex~\(\III=\bigoplus_{p=0}^{r}I^{p}[p]\)
  and consider the double complex
  \begin{equation}
    C = \Hom_{R}(M,\III)
   \quad\text{with}\quad
    C^{pq} = \Hom_{R}(M,I^{p})^{q}. 
  \end{equation}
  Filtering by \(q\)-degree and using the freeness of~\(M\), we see
  that the canonical injection~\(\Hom_{R}(M,R)\to C\)
  is a quasi-isomorphism.
  Filtering by \(p\)-degree instead leads to a spectral sequence
  with \(E_{1}^{p}=\Hom_{R}(H^{*}(M),I^{p})\), hence with the desired \(E_{2}\)~page.

  This also establishes the existence of a spectral of the second type
  converging to the cohomology of the \(R\)-dual of~\(\hCT_{*}(A,B)\). 
  Since \(H^{*}(A,B)\) is a finite-dimensional vector space,
  another spectral sequence argument shows
  that the canonical inclusion
  \begin{equation*}
    \CT^{*}(A,B) \to \Hom_{R}(\hCT_{*}(A,B),R)
  \end{equation*}
  is a quasi-isomorphism.
  (Note that in this part we are using Remark~\ref{rem:finiteness}
  to ensure naive convergence of all spectral sequences.)
\end{proof}

\begin{example}
  \label{ex:homogeneous-space}
  Consider the homogeneous space~\(X=T/T'\), where \(T'\) is a subtorus of rank~\(r-i\).
  Then \(\HT^{*}(X)=H^{*}(BT')=\vcentcolon R'\), hence
  \begin{equation}
    \Ext_{R}^{j}(\HT^{*}(X),R)=
    \begin{cases}
      R'[-2i] & \text{if \(j=i\),} \\
      0 & \text{otherwise}
    \end{cases}
  \end{equation}
  by a computation similar to Lemma~\ref{thm:Ext-Koszul}.
  The universal coefficient spectral sequence therefore degenerates
  to a single column, and
  \begin{equation}
    \hHT_{*}(X) = R'[-i].
  \end{equation}
  Note that both \(\HT^{*}(X)\)~and~\(\hHT_{*}(X)\) are Cohen--Macaulay \(R\)-modules
  of dimension~\(r-i\).
  A generalization of this example in Section~\ref{sec:orbit-filtration}
  will be a crucial ingredient for our main result.
\end{example}

\subsection{Poincaré duality}

We say that \(X\) is a \emph{\(\kk \)-Poincaré duality space}
(\emph{PD~space}\/ for short) of formal dimension~\(n\)
if \(X\) is non-empty and connected and
if there is a distinguished class~\(o\in H_{n}(X)\), called an orientation,
such that the Poincaré pairing (of degree~\(-n\))
\begin{equation}
  \label{eq:poincare-cup}
  H^{*}(X)\times H^{*}(X) \to \kk ,
  \quad
  (\alpha,\beta) \mapsto
  \pair{\alpha\cup\beta,o}
\end{equation}
is non-degenerate. Equivalently, this pairing is perfect,
which means that the induced map
\begin{equation}
  \label{eq:poincare-cap}
  H^{*}(X) \to \Hom_{\kk }(H^{*}(X),\kk ) = H_{*}(X),
  \quad
  \alpha\mapsto \alpha\cap o
\end{equation}
is an isomorphism of vector spaces (of degree~\(-n\)).
For example, compact oriented \(\kk \)-homology manifolds are PD~spaces,
\cf~\cite[\S 3.6]{Borel:1957}.

In the same way we can dualize
the equivariant cup product 
to a cap product. On the cochain level it is given by
\begin{equation}
  \CT^{*}(X)\otimes\hCT_{*}(X) \xrightarrow{\cap} \hCT_{*}(X),
  \quad
  \pair{\alpha,\beta \cap b} = \pair{\alpha\cup\beta,b}
\end{equation}
for \(\alpha\),~
\(\beta\in\CT^{*}(X)\)~and~\(b\in\hCT_{*}(X)\).

\begin{proposition} 
  \label{thm:equiv-PD}
  Let \(X\) be a PD~space of formal dimension~\(n\).
  The orientation~\(o\) lifts uniquely to an equivariant orientation~\(o_{T}\in \hHT_{n}(X)\)
  under the restriction map 
  \( 
    \hHT_{*}(X) \to H_{*}(X)
  \). 
  Moreover, taking the cap product with~\(o_{T}\) gives an isomorphism of \(R\)-modules
  (of degree~\(-n\))
  \begin{equation*}
    \HT^{*}(X) \xrightarrow{\cap o_{T}} \hHT_{*}(X).
  \end{equation*}
\end{proposition}

\begin{proof}
  The map~\(\hHT_{*}(X) \to H_{*}(X)\)
  is the edge homomorphism of the spectral sequence~\eqref{eq:ss-hHT-E2}.
  Hence the first claim holds because
  \(H_{n}(X)\otimes 1\) is the only term of total degree~\(-n\) in the \(E_{1}\)~page.

  Let \(c_{T}\in \hCT_{n}(X)\) be a representative of~\(o_{T}\),
  so that its restriction to~\(C_{n}(X)\) is a representative of~\(o\).
  Consider the morphism of dg~\(R\)-modules
  \begin{equation*}
    \CT^{*}(X) \xrightarrow{\cap c_{T}} \hCT_{*}(X)
  \end{equation*}
  and filter both sides by \(R\)-degree.
  It follows that the induced map between the \(E_{1}\)~pages
  of the associated spectral sequences
  is the \(R\)-linear extension of the non-equivariant Poincaré duality map,
  hence an isomorphism.
\end{proof}

For a different proof of Proposition~\ref{thm:equiv-PD}
which uses the minimal Hirsch--Brown model in an essential way,
see \cite[pp.~352--353]{AlldayPuppe:1993}.

\begin{remark}
  \label{rem:equiv-PD}
  As already observed in the introduction,
  equivariant Poincaré duality
  does not necessarily translate into non-degeneracy or perfection
  of the equivariant Poincaré pairing
  \begin{equation}
    \label{eq:poincare-cup-equiv}
    \HT^{*}(X)\times \HT^{*}(X) \to R,
    \quad
    (\alpha,\beta) \mapsto
    \pair{\alpha\cup\beta,o_{T}}.
  \end{equation}
  For example, if \(\HT^{*}(X)\) is a torsion \(R\)-module (which by the localization theorem
  means that there are no fixed points), then the pairing~\eqref{eq:poincare-cup-equiv}
  is trivial. This happens for instance in Example~\ref{ex:homogeneous-space} unless \(T'=T\).
  We will come back to this point in Section~\ref{sec:applications-PD}.
\end{remark}

\section{The main result}
\label{sec:main-result}

\subsection{The orbit filtration}
\label{sec:orbit-filtration}

For~\(-1\le i\le r\), we define the \emph{equivariant \(i\)-skeleton}~\(X_{i}\) of~\(X\)
to be the union of orbits of dimension at most~\(i\).
In particular, \(X_{-1}=\emptyset\), \(X_{0}=X^{T}\) is the fixed point set, and \(X_{r}=X\).
All~\(X_{i}\) are closed in~\(X\).

\begin{assumption}
  \label{ass:orbit-filtration}
  From now on 
  we assume that the characteristic of our ground field~\(\kk \) is~\(0\).
  In addition to the assumptions stated in Section~\ref{sec:assumptions}
  and Assumption~\ref{ass:H-finite},
  we require that \(X\) be finite-dimensional and that
  the set~\(\{T_{x}^{0} : x\in X\}\) be finite, where \(T_{x}^{0}\) denotes the identity component
  of the isotropy group~\(T_{x}\).
  Moreover, all~\(X_{i}\) are assumed to be locally contractible.
\end{assumption}

By~\cite[Prop.~4.1.14]{AlldayPuppe:1993} all~\(H^{*}(X_{i})\) are finite-dimensional, so that
all pairs~\((X_{i},X_{j})\), \(i\ge j\) satisfy our assumptions on \(T\)-pairs
as stated in Section~\ref{sec:assumptions} and Assumption~\ref{ass:H-finite}. 
See Remark~\ref{rem:loc-contractible} below for a possible weakening of these assumptions.  

\begin{proposition}
  \label{thm:stratum-cm}
  The \(R\)-modules \(\HT^{*}(X_{i},X_{i-1})\)~and~\(\hHT_{*}(X_{i},X_{i-1})\) are zero or
  Cohen--Macaulay of dimension~\(r-i\) for~\(0\le i\le r\).
\end{proposition}

This Cohen--Macaulay property
was already crucial
for Atiyah~\cite[Lecture~7]{Atiyah:1974}, hence also
for Bredon~\cite{Bredon:1974} and Franz--Puppe~\cite{FranzPuppe:2011}.

\begin{proof}
  Assume first that for all~\(x\in Y=X_{i}\setminus X_{i-1}\) the identity component~\(T_{x}^{0}\)
  is a fixed subtorus~\(T'\subset T\) and choose a torus complement~\(T''\subset T\).
  This gives a decomposition~\(R=H^{*}(BT')\otimes H^{*}(BT'')=:R'\otimes R''\).
  By tautness 
  and excision,
  \begin{align*}
    \HT^{*}(X_{i},X_{i-1}) &= \dirlim \HT^{*}(X_{i},U)
    = \dirlim \HT^{*}(Y,Y\cap U), \\
    \intertext{the direct limit being taken over all \(T\)-invariant open neighbourhoods~\(U\) of~\(X_{i-1}\) in~\(X_{i}\),
    which are cofinal among all open neighbourhoods.
    Using that \(\kk \) is a field of characteristic~\(0\) and that \(T''\) acts with finite isotropy, we get}
    &= \dirlim R' \otimes H_{T''}^{*}(Y,Y\cap U)
    = \dirlim R' \otimes H^{*}(Y/T'',(Y\cap U)/T'') \\
    &= R' \otimes H^{*}(X_{i}/T'',X_{i-1}/T'').
  \end{align*}
  (Here $H^{*}(-)$ denotes Alexander--Spanier cohomology because the orbit spaces 
  may fail to satisfy our assumptions on spaces.)
  The \(R''\)-action on  \(H^{*}(X_{i}/T'',X_{i-1}/T'')\) need not be trivial, but
  because this module 
  is finite-dimensional over~\(\kk \),
  there is a finite filtration of~\(\HT^{*}(X_{i},X_{i-1})\) such that
  each successive quotient is free over~\(R'\) with trivial \(R''\)-action.
  Such a quotient clearly is Cohen--Macaulay of dimension \(r-i\), hence
  so is \(\HT^{*}(X_{i},X_{i-1})\) by the long exact sequence for~\(\Ext\).
  Moreover, by the universal coefficient theorem (Proposition~\ref{thm:uct}),
  \begin{align*}
    \hHT_{*}(X_{i},X_{i-1}) &= R' \otimes \Ext_{R''}^{i}(H^{*}(X_{i}/T'',X_{i-1}/T''),R'')[i] \\
    &= R' \otimes H_{*}(X_{i}/T'',X_{i-1}/T'')[-i]
  \end{align*}
  is of the same algebraic type, so that this module is Cohen--Macaulay
  of dimension~\(r-i\) as well.
  (\Cf~Example~\ref{ex:homogeneous-space}.)

  In the general case,
  \(X_{i}\setminus X_{i-1}\) is
  the disjoint union of finitely many spaces~\(Y_{\alpha}\), \(\alpha\in A_{p}\)
  such that \(T_{x}^{0}= T_{\alpha}\) for all~\(x\in Y_{\alpha}\) and some
  subtorus~\(T_{\alpha}\subset T\) of rank~\(r-i\).
  Hence
  \begin{equation*}
    \HT^{*}(X_{i},X_{i-1}) = \bigoplus_{\alpha\in A_{p}} \HT^{*}(X_{i},X_{i}\setminus Y_{\alpha}),
  \end{equation*}
  and we can conclude by the same reasoning as before.
\end{proof}

We record a corollary of the preceding proof
for later use in Section~\ref{sec:partial-exactness}.
Part~\eqref{gg2} is a special case of the
localization theorem in equivariant cohomology.

\begin{lemma}
  \label{thm:localization-orbit-filtration}
  Let \(S\subset R\) be a multiplicative subset and \(0\le i\le r\).
  For~\(x\in X\), let \(J_{x}\) be the kernel of the projection~\(R\to H^{*}(BT_{x})\).
  \begin{enumerate}
  \item \label{gg1} If \(S\cap J_{x}=\emptyset\) for all~\(x\in X_{i}\setminus X_{i-1}\),
    then \(\HT^{*}(X_{i},X_{i-1})\) has no \(S\)-torsion.
  \item \label{gg2} If \(S\cap J_{x}\neq\emptyset\) for all~\(x\in X_{i}\setminus X_{i-1}\),
    then \(S^{-1}\HT^{*}(X_{i},X_{i-1})=0\).
  \end{enumerate}
\end{lemma}

\begin{proof}
  We have just seen that \(\HT^{*}(X_{i},X_{i-1})\) has a finite filtration
  by \(R\)-modules of the form~\(H^{*}(BT_{x}^{0})\) for some~\(x\in X_{i}\setminus X_{i-1}\).
  In the first case, any~\(a\in S\) acts injectively on each piece of the filtration,
  hence on~\(\HT^{*}(X_{i},X_{i-1})\).
  In the second case, each piece of the filtration is annihilated by some~\(a\in S\).
  The product of these elements then annihilates \(\HT^{*}(X_{i},X_{i-1})\).
\end{proof}

The filtration of~\(X\) by equivariant skeletons leads to a spectral sequence
converging to~\(\HT^{*}(X)\). The \(E_{1}\)~page is the non-augmented Atiyah--Bredon sequence 
\begin{equation}
  \label{eq:atiyah-bredon-truncated}
  \let\longrightarrow\rightarrow
  \HT^{*}(X_0)
  \longrightarrow \HT^{*}(X_1, X_0)[-1]
  \longrightarrow \cdots
  \longrightarrow \HT^{*}(X_r, X_{r-1})[-r],
\end{equation}
denoted by~\(\AB^{*}(X)\),
and the \(E_{2}\)~page therefore \(H^{*}(\AB^{*}(X))\). The corresponding spectral sequence
for equivariant homology is much easier to understand:

\begin{corollary}
  \label{thm:hHT-orbit-degeneration}
  The spectral sequence associated with the orbit filtration of~\(\hCT_{*}(X)\)
  and converging to~\(\hHT_{*}(X)\)
  degenerates at~\(E^{1}_{p}=\hHT_{*}(X_{p},X_{p-1})\).
\end{corollary}

\begin{proof}
  By Proposition~\ref{thm:stratum-cm}, \(E^{1}_{p}\) is zero or Cohen--Macaulay of dimension~\(r-p\).
  Lemma~\ref{thm:CM-map-0} therefore rules out
  any non-zero higher differential~\(d^{s}\colon E^{1}_{p}\to E^{1}_{p-s}\). 
\end{proof}

\begin{remark}
  \label{rem:duflot}
  By extending Poincaré--Alexander--Lefschetz duality to the equivariant setting,
  Corollary~\ref{thm:hHT-orbit-degeneration} can be related
  to a result of J.~Duflot~\cite[Thm.~1]{Duflot:1983}
  involving the equivariant cohomology
  of the complements~\(X\setminus X_{i}\) of the orbit filtration.
  We will elaborate on this in future work.
\end{remark}

\begin{corollary}
  \label{thm:Ext-hHT-zero}
  \(\Ext_{R}^{q}(\hHT_{*}(X_{p},X_{p'}),R)=0\)  for~\(q\le p'\) or~\(q>p\).
\end{corollary}

\begin{proof}
  By Corollary~\ref{thm:hHT-orbit-degeneration}, \(\hHT_{*}(X_{p},X_{p'})\) has a filtration
  whose successive quotients are of the form~\(\hHT_{*}(X_{i},X_{i-1})\), \(p\le i<p'\).
  The claim therefore follows from Proposition~\ref{thm:stratum-cm} and the long exact sequence
  for~\(\Ext\).
\end{proof}

\begin{remark}
  \label{rem:loc-contractible}
  For a continuous torus action on a space~\(X\) that is not known to be a
  \(T\)-CW~complex it is not true, in general, that each~\(X_i\) is locally 
  contractible even when \(X\) is a topological manifold. One can easily avoid 
  this technical difficulty as follows.

  The Alexander--Spanier cohomology of a pair~\((A,B)\) of closed invariant subspaces
  can be expressed as
  \begin{equation}
    \veeHT^{*}(A, B) = \dirlim \HT^{*}(U, V)
  \end{equation}
  where \((U, V)\) ranges over the open invariant
  neighbourhood pairs of~\((A, B)\). So one defines
  \begin{align}
    \veeCT^{*}(A, B) &= \dirlim \CT^{*}(U, V) \\
    \intertext{and}
    \veehCT_{*}(A, B) &= \Hom_{R}(\veeCT^{*}(A, B), R).
  \end{align}
  The proofs in
  this paper then proceed in exactly the same way, because the
  localization theorem requires Alexander--Spanier cohomology in this
  generality. (See, \eg, \cite[Ex.~3.2.11]{AlldayPuppe:1993} for details.)
  Thus the results of this paper hold using this version of equivariant
  Alexander--Spanier cohomology and homology if \(X\) is Hausdorff, second-countable,
  locally compact and locally contractible even if some of the equivariant
  skeletons are not locally contractible.
\end{remark}

\subsection{Comparing the spectral sequences}
\label{sec:comparison-ss}

The ground field~\(\kk\) is still assumed to be of characteristic~\(0\), and
the assumptions on the orbit filtration~\((X_{i})\) stated in Section~\ref{sec:orbit-filtration}
remain in force.

The aim of this and the following sections is to prove the following result:

\begin{theorem}
  \label{thm:exthab-ss}
  For any \(T\)-space~\(X\), 
  the following two spectral sequences converging to~\(\HT^{*}(X)\) are naturally isomorphic from the \(E_{2}\)~page on:
  \begin{enumerate}
  \item The one induced by the orbit filtration with~\(E_{1}^{p}=\HT^{*}(X_{p},X_{p-1})\),
  \item The universal coefficient spectral sequence with~\(E_{2}^{p}=\Ext_{R}^{p}(\hHT_{*}(X),R)\).
  \end{enumerate}
\end{theorem}

\begin{remark}
  The assumptions on~\(\kk\) and on the filtration~\((X_{i})\)
  are only required to apply Proposition~\ref{thm:stratum-cm}
  and Corollary~\ref{thm:hHT-orbit-degeneration}.
  All results in Sections \ref{sec:orbit-filtration}~and~\ref{sec:comparison-ss}
  are valid more generally for any field~\(\kk\) and any filtration~\((X_{i})\), \(1\le i\le r\),
  satisfying our assumption on \(T\)-spaces
  and such that \(\HT^{*}(X_{i},X_{i-1})\)~and~\(\hHT_{*}(X_{i},X_{i-1})\) are zero or Cohen--Macaulay of dimension~\(r-i\).
  We will use this elsewhere to treat actions of~\(p\)-tori~\((\Z_{p})^{r}\), which need some further considerations.
\end{remark}


To begin with the proof, let
\begin{equation}
  \label{resolution-I}
  0 \longrightarrow R
    \longrightarrow I^{0}
    \longrightarrow I^{1}
    \longrightarrow \cdots
    \longrightarrow I^{r}
    \longrightarrow 0
\end{equation}
be the minimal (graded) injective resolution of~\(R\), \cf~\cite[Sec.~3.6]{BrunsHerzog:1998}.
Recall that this resolution is constructed inductively
by setting \(M^{0}=R\), \(I^{i}\) equal to the (graded) injective hull of~\(M^{i}\)
and \(M^{i+1}=I^{i}/M^{i}\). In particular, there are short exact sequences
\begin{equation}\label{M-I-M}
  0 \longrightarrow M^{i}
    \longrightarrow I^{i}
    \longrightarrow M^{i+1}
    \longrightarrow 0.
\end{equation}
Let \(\III^{p}=I^{p}[p]\) and \(\III=\bigoplus_{p=0}^{r}\III^{p}\) be the total complex associated to the injective resolution~\eqref{resolution-I}.

On the dg~\(R\)-module~%
\(
  \Hom_{R}(\hCT_{*}(X),\III)
\)
we introduce two decreasing filtrations \(\OO\)~and~\(\II\) with
\begin{subequations}
\begin{align}
  \label{eq:def-OO}
  \OO^{p} &= \Hom_{R}(\hCT_{*}(X,X_{p-1}),\III), \\
  \label{eq:def-II}
  \II^{p} &= \Hom_{R}(\hCT_{*}(X),\III^{\ge p}),
\intertext{where the subcomplex~\(\III^{\ge p}=\bigoplus_{i\ge p}\III^{i}\) of~\(\III\)
is the total complex of the minimal injective resolution of~\(M^{p}[p]\).
Because we cannot compare these filtrations 
directly, we additionally consider the ``diagonal'' filtration~\(\DD\) with}
  \label{eq:def-DD}
  \DD^{p} &= \OO^{p}\cap \II^{p} = \Hom_{R}(\hCT_{*}(X,X_{p-1}),\III^{\ge p}).
\end{align}
\end{subequations}
We are going to show
in Sections \ref{sec:comparing-D-O}~and~\ref{sec:comparing-D-I}
that the maps of spectral sequences
associated to the inclusions \(\DD\to\OO\) and \(\DD\to\II\)
become isomorphisms from, respectively, the \(E_{1}\)~and~\(E_{2}\) pages on.

The orbit filtration also induces a decreasing filtration on~\(\CT^{*}(X)\),
which by abuse of notation we denote by~\(\OO\) as well.

\begin{lemma}
  \label{quiso-Hom-R-I}
  The canonical map
  \begin{equation*}
    \CT^{*}(X) \to \Hom_{R}(\hCT_{*}(X),\III)
  \end{equation*}
  is a quasi-isomorphism preserving the filtrations~\(\OO\).
  Moreover, the associated map of spectral sequences
  is an isomorphism from the \(E_{1}\)~page on.
\end{lemma}

\begin{proof}
  The map above is clearly filtration-preserving.
  Because \(\III\) is the direct sum of injective modules,
  the induced map on the \(E_{0}\)~page is
  \begin{equation*}
    \CT^{*}(X_{p},X_{p-1}) \to \Hom_{R}(\hCT_{*}(X_{p},X_{p-1}),\III).
  \end{equation*}
  This map is a quasi-isomorphism, \cf~the proof of Proposition~\ref{thm:uct}.
\end{proof}

Note that
\begin{equation}
  E_{2}^{p}(\II)=\Ext_{R}^{p}(\hHT_{*}(X),R)
\end{equation}
because each module~\(I^{p}\) is injective.
Together with the comparisons of \(\DD\),~\(\OO\) and~\(\II\) below
this completes the proof of Theorem~\ref{thm:exthab-ss}
as the naturality in~\(X\) is clear by construction.
To prepare for the comparisons, we state two lemmas.

\begin{lemma}\label{X-Xp-Iq-zero}
  \(\Hom_{R}(\hHT_{*}(X,X_{p}),I^{q})=0\) for~\(p\ge q\).
\end{lemma}

\begin{proof}
  Assume that \(f\in\Hom_{R}(\hHT_{*}(X,X_{p}),I^{q})\) is non-zero.
  Because \(I^{q}\) is the graded injective hull of~\(M^{q}\),
  there is an~\(a\in R\) such that \(a f(y)\in M^{q}\) for all \(y\in\hHT_{*}(X,X_{p})\)
  and \(0\ne af\in\Hom_{R}(\hHT_{*}(X,X_{p}),M^{q})\).
  But this is impossible:
  Using Corollary~\ref{thm:Ext-hHT-zero} and
  the long exact \(\Ext\)~sequence coming from~\eqref{M-I-M},
  one can show by induction on~\(q\) that
  \begin{equation*}
    \Hom_{R}(\hHT_{*}(X,X_{p}),M^{q}) = \Ext_{R}^{1}(\hHT_{*}(X,X_{p}),M^{q-1})
    = \cdots 
    = \Ext_{R}^{q}(\hHT_{*}(X,X_{p}),R)
  \end{equation*}
  vanishes for~\(p\ge q\).
\end{proof}

\begin{lemma}
  \label{thm:Ext-Hom-M}
  The inclusion~\(M^{p}\to I^{p}\) induces a quasi-isomorphism
  \begin{equation*}
    \Hom_{R}(\hHT_{*}(X_{p},X_{p-1}),M^{p}[p]) \to \Hom_{R}(\hHT_{*}(X_{p},X_{p-1}),\III).
  \end{equation*}
\end{lemma}

\begin{proof}
  The \(R\)-module~%
  \(\Hom_{R}(\hHT_{*}(X_{p},X_{p-1}),\III^{<p})\) vanishes by Lemma~\ref{X-Xp-Iq-zero},
  and so does \(\Ext_{R}^{q}(\hHT_{*}(X_{p},X_{p-1}),R)\) for~\(q\ne p\)
  by Proposition~\ref{thm:stratum-cm}.
  As \(M^{p}[p]\to\III^{\ge p}\) is an injective resolution,
  we find
  \(\Ext_{R}^{p}(\hHT_{*}(X_{p},X_{p-1}),R)=\Hom_{R}(\hHT_{*}(X_{p},X_{p-1}),M^{p})\).
\end{proof}

\subsection{Comparing \texorpdfstring{\(\DD\)}{D}~and~\texorpdfstring{\(\OO\)}{O}}
\label{sec:comparing-D-O}

Recall that each~\(\III^{p}=I^{q}[q]\) is an injective \(R\)-module, so that \(\Hom_{R}(-,\III^{q})\) is an exact functor
and hence commutes with taking homology.
We will use this throughout the proof without mentioning.
For example, there is a short exact sequence
\begin{multline}
  0 \to
  \Hom_{R}(\hCT_{*}(X,X_{p}),\III^{q}) \to
  \Hom_{R}(\hCT_{*}(X,X_{p-1}),\III^{q}) \\ \to
  \Hom_{R}(\hCT_{*}(X_{p},X_{p-1}),\III^{q}) \to
  0
\end{multline}
(which in this case also follows from the freeness of equivariant chain complexes over~\(R\)).
From this and the definition of the filtration~\(\OO\) we get
the second of the following two \(E_{0}\)~pages. The first follows by additionally
observing that elements in~\(\DD^{p+1}\) do not map to~\(\III^{p}\).
\begin{align}
  \label{eq:E0-DD}
  E_{0}^{p}(\DD)[p] &= \Hom_{R}(\hCT_{*}(X,X_{p-1}),\III^{p})
    \oplus 
    \Hom_{R}(\hCT_{*}(X_{p},X_{p-1}),\III^{>p}), \\
  E_{0}^{p}(\OO)[p] &= \Hom_{R}(\hCT_{*}(X_{p},X_{p-1}),\III).
\end{align}
Filtering both \(E_{0}\)~pages by \(I\)-degree gives spectral sequences converging
to \(E_{1}^{p}(\DD)\) and~\(E_{1}^{p}(\OO)\), respectively.
For the \(E_{1}\)~pages of these intermediate spectral sequences
we obtain
\begin{align}
  \label{eq:E1-E0-DD-II}
  E_{1}^{q}E_{0}^{p}(\DD)[p] &=
    \begin{cases}
      0 & \text{if \(q<p\),} \\
      \Hom_{R}(\hHT_{*}(X,X_{p-1}),I^{p}) & \text{if \(q=p\),} \\
      \Hom_{R}(\hHT_{*}(X_{p},X_{p-1}),I^{q}) & \text{if \(q>p\),}
    \end{cases} \\
  \label{eq:E1-E0-OO-II}
  E_{1}^{q}E_{0}^{p}(\OO)[p] &= \Hom_{R}(\hHT_{*}(X_{p},X_{p-1}),I^{q}) \\
    \notag &= 
    \begin{cases}
      0 & \text{if \(q<p\),} \\
      \Hom_{R}(\hHT_{*}(X_{p},X_{p-1}),I^{q}) &  \text{if \(q\ge p\),}
    \end{cases}
\end{align}
where the last identity follows from Lemma~\ref{X-Xp-Iq-zero}.
Applying \(\Hom_{R}(-,I^{p})\) to the long exact sequence for the triple~\((X,X_{p},X_{p-1})\)
and observing again Lemma~\ref{X-Xp-Iq-zero}, we find
that the map
\begin{equation}
  \Hom_{R}(\hHT_{*}(X,X_{p-1}),I^{p})
  \to \Hom_{R}(\hHT_{*}(X_{p},X_{p-1}),I^{p})
\end{equation}
is an isomorphism,
hence so is the map~\(E_{1}^{p}E_{0}^{q}(\DD)\to E_{1}^{p}E_{0}^{q}(\OO)\)
and therefore \(E_{1}^{q}(\DD)\to E_{1}^{q}(\OO)\) as well.

\subsection{Comparing \texorpdfstring{\(\DD\)}{D}~and~\texorpdfstring{\(\II\)}{I}}
\label{sec:comparing-D-I}

In addition to~\eqref{eq:E0-DD}, which we now write as
\begin{align}
  E_{0}^{q}(\DD)[q] &= \Hom_{R}(\hCT_{*}(X,X_{q-1}),\III^{q})
  \oplus \Hom_{R}(\hCT_{*}(X_{q},X_{q-1}),\III^{>q}), \\
\intertext{we have}
  E_{0}^{q}(\II)[q] &= \Hom_{R}(\hCT_{*}(X),\III^{q}).
\end{align}
The filtration~\(\OO\) induces filtrations on both \(E_{0}^{q}(\DD)\) and \(E_{0}^{q}(\II)\). 
This gives spectral sequences converging
to \(E_{1}^{q}(\DD)\) and~\(E_{1}^{q}(\II)\), respectively,
with \(E_{0}\)~pages
\begin{align}
  E_{0}^{p}E_{0}^{q}(\DD)[p+q] &= \begin{cases}
    0 & \text{if \(p<q\),} \\
    \Hom_{R}(\hCT_{*}(X_{p},X_{p-1}),\III^{\ge p}) & \text{if \(p=q\),} \\
    \Hom_{R}(\hCT_{*}(X_{p},X_{p-1}),\III^{q}) & \text{if \(p>q\),}
  \end{cases} \\
  E_{0}^{p}E_{0}^{q}(\II)[p+q] &= \Hom_{R}(\hCT_{*}(X_{p},X_{p-1}),\III^{q}).
\end{align}

For the second spectral sequence, we have
\begin{equation}
  E_{1}^{p}E_{0}^{q}(\II)[p+q] = \Hom_{R}(\hHT_{*}(X_{p},X_{p-1}),\III^{q}).
\end{equation}
Corollary~\ref{thm:hHT-orbit-degeneration} now implies that this spectral sequence degenerates,
\(E_{1}^{p}E_{0}^{q}(\II)=E_{\infty}^{p}E_{0}^{q}(\II)\Rightarrow E_{1}^{p}(\II)\).

For the first spectral sequence, we claim that
\begin{equation}
  E_{1}^{p}E_{0}^{q}(\DD)[p+q] = \begin{cases}
    \Hom_{R}(\hHT_{*}(X_{p},X_{p-1}),M^{p}[p]) & \text{if \(p=q\),} \\
    0 & \text{otherwise.}
  \end{cases}
\end{equation}
The claim is trivial for~\(p<q\), and
Lemma~\ref{X-Xp-Iq-zero} proves it for~\(p>q\).
For the case~\(p=q\), we filter \(\Hom_{R}(\hHT_{*}(X_{p},X_{p-1}),\III^{\ge p})\)
by \(I\)-degree.
This leads to an intermediate spectral sequence
whose \(E_{1}\)~page
\begin{equation}
  E_{1}^{s} =
  \begin{cases}
    \Hom_{R}(\hHT_{*}(X_{p},X_{p-1}),\III^{s}) & \text{if \(s\ge p\),} \\
    0 & \text{otherwise}
  \end{cases}
\end{equation}
can in fact be written
in the first form
for all~\(s\) by Lemma~\ref{X-Xp-Iq-zero}.
Hence, the claim follows from Lemma~\ref{thm:Ext-Hom-M},
and the original spectral sequence 
degenerates at the \(E_{1}\)~level, too:
\(E_{1}^{p}E_{0}^{q}(\DD)=E_{\infty}^{p}E_{0}^{q}(\DD)\Rightarrow E_{1}^{q}(\DD)\).

So far, we know the associated graded modules of~\(E_{1}(\DD)\)~and~\(E_{1}(\II)\)
induced by the filtration~\(\OO\). Because the filtration of~\(E_{1}(\DD)\)
is compatible with the differential~\(d^{1}\), it gives rise to another spectral sequence
starting at~\(E_{0}E_{1}(\DD)=E_{\infty}E_{0}(\DD)\) and converging to~\(E_{2}(\DD)\),
and similarly for~\(\II\). The map \(E_{0}^{p}E_{1}(\DD)\to E_{0}^{p}E_{1}(\II)\)
is the quasi-isomorphism
\begin{equation}
  \Hom_{R}(\hHT_{*}(X_{p},X_{p-1}),M^{p}[p]) \to \Hom_{R}(\hHT_{*}(X_{p},X_{p-1}),\III)
\end{equation}
from Lemma~\ref{thm:Ext-Hom-M}.
This forces \(E_{\infty}^{p}E_{1}(\DD)=E_{\infty}^{p}E_{1}(\II)\),
hence \(E_{2}(\DD)=E_{2}(\II)\). In other words, the
spectral sequences for the filtrations \(\DD\)~and~\(\II\)
are isomorphic from the \(E_{2}\)~page on.

\section{Consequences and applications}
\label{sec:applications}

Note that Assumption~\ref{ass:orbit-filtration} remains in force.

\subsection{Immediate consequences}

Theorem~\ref{thm:exthab-ss} implies in particular
that the \(E_{2}\) pages of the two spectral sequences coincide.
Due to its importance, we state this result separately.

\begin{theorem}
  \label{thm:exthab}
  For any~\(j\ge0\), the \(j\)-th cohomology of the (non-augmented) Atiyah--Bredon sequence
  is 
  \begin{equation*}
    H^{j}(\AB^{*}(X)) = \Ext_{R}^{j}(\hHT_{*}(X),R).
  \end{equation*}
  Under this isomorphism, the map~\(\HT^{*}(X)\to H^{0}(\AB^{*}(X))\)
  corresponds to the
  canonical map~\(\HT^{*}(X)\to\Hom_{R}(\hHT_{*}(X),R)\).
\end{theorem}

\begin{remark}
  It is crucial for Theorem~\ref{thm:exthab}
  that the \(T\)-space~\(X\) satisfies
  Assumption~\ref{ass:orbit-filtration}.
  For example, the conclusion does not hold
  if one replaces an allowed \(T\)-space~\(X\)
  by~\(ET\times X\) 
  unless \(T\) acts locally freely on~\(X\).
\end{remark}

\begin{corollary}
  \label{thm:HT-dual-to-hHT}
  The Chang--Skjelbred sequence~\eqref{eq:chang-skjelbred}
  is exact if and only if the canonical map
  \( 
    \HT^{*}(X) \to \Hom_{R}(\hHT_{*}(X),R)
  \) 
  is an isomorphism.
\end{corollary}

\begin{corollary}
  \label{thm:quasi-iso-same-HAB}
  Let \(X\)~and~\(X'\) be two \(T\)-spaces having isomorphic equivariant homology.
  Then \(H^{*}(\AB^{*}(X))\cong H^{*}(\AB^{*}(X'))\).
\end{corollary}

Note that Corollary~\ref{thm:quasi-iso-same-HAB} applies
if there exists an equivariant map~\(X\to X'\) inducing an
isomorphism in non-equivariant cohomology, hence also
in equivariant homology by Remark~\ref{rem:finiteness}.

\begin{remark}
  \label{rem:local-duality}
  \def\HH{\mathcal{H}^{T}}
  Let \(M\) be a finitely generated \(R\)-module and
  \(H_{\m}^{*}(M)\) its local cohomology
  with respect to the maximal ideal~\(\m\lhd R\).
  Local duality in this case refers to an isomorphism of \(R\)-modules
  \begin{equation}
    \label{eq:local-duality}
    H_{\m}^{r-j}(M) = \Hom_{\kk }\bigl(\Ext_{R}^{j}(M,R[r]), \kk \bigr),
  \end{equation}
  \cf~\cite[Thm.~A1.9]{Eisenbud:2005}.
  
  Let \(\HH_{*}(A,B)=H_{*}(A_{T},B_{T})\) denote the \emph{homology}
  of the Borel construction of the \(T\)-pair~\((A,B)\) and
  \(H_{j}(\AB_{*}(X))=\Hom_{\kk }(H^{j}(\AB^{*}(X)),\kk )\)
  the homology of the non-augmented homological Atiyah--Bredon-sequence
  \begin{equation}
    \let\longrightarrow\leftarrow
    \HH_*(X_0)
    \longrightarrow \HH_{*+1}(X_1, X_0)
    \longrightarrow \cdots
    \longrightarrow \HH_{*+r}(X_r, X_{r-1}).
  \end{equation}
  Then, by local duality, the isomorphism~\(H^{j}(\AB^{*}(X))=\Ext_{R}^{j}(\hHT_{*}(X),R)\)
  is equivalent to an isomorphism of \(R\)-modules
  \begin{equation}
    H_{j}(\AB_{*}(X)) = H_{\m}^{r-j}(\hHT_{*}(X))[r].
  \end{equation}
\end{remark}

\subsection{Partial exactness}
\label{sec:partial-exactness}

In this section we characterize when
a front piece of the Atiyah--Bredon sequence
\begin{multline}\label{eq:atiyah-bredon-local}
  \let\longrightarrow\rightarrow
  0
  \longrightarrow \HT^{*}(X)
  \longrightarrow \HT^{*}(X_0)
  \longrightarrow \HT^{*}(X_1, X_0)[-1]
  \longrightarrow \\ \cdots
  \let\longrightarrow\rightarrow
  \longrightarrow \HT^{*}(X_r, X_{r-1})[-r]
  \longrightarrow 0
\end{multline}
is exact. We write \(\barAB^{*}(X)\) for this complex of \(R\)-modules
(with \(\barAB^{-1}(X)=\HT^{*}(X)\)),
in contrast to the non-augmented Atiyah--Bredon complex~\(\AB^{*}(X)\)
introduced in~\eqref{eq:atiyah-bredon-truncated}.

\begin{lemma}
  \label{thm:exact-2}
  The Atiyah--Bredon complex~\(\barAB^{*}(X)\) is the \(E_{1}\)~page
  of a spectral sequence converging to~\(0\).
  In particular, if it is exact
  at all but possibly two adjacent terms, then it is exact everywhere.
  The analogous statement holds if one localizes \(\barAB^{*}(X)\)
  with respect to any multiplicative set~\(S\subset R\).
\end{lemma}

\begin{proof}
  Let \(Y=C X\) be the cone over~\(X\) with apex~\(*\), and
  consider the filtration~\(Y_{i}=X\cup C X_{i}\), \(-1\le i\le r\).
  The complex~\(S^{-1}\barAB^{*}(X)\)
  is, up to degree shift, the
  \(E_{1}\)~page of the spectral sequence associated to this filtration
  and converging to~\(S^{-1}\HT^{*}(Y,*)=0\). Since higher differentials
  cannot connect adjacent columns, non-zero terms in only one or two adjacent
  columns of the \(E_{2}\)~page would lead to~\(E_{\infty}\ne0\), a contradiction.
\end{proof}

\begin{theorem}
  \label{thm:conditions-partial-exactness}
  The following conditions are equivalent for any~\(0\le j\le r\):
  \begin{enumerate}
  \item \label{q1} The Atiyah--Bredon sequence~\eqref{eq:atiyah-bredon-local}
    is exact at all positions~\(-1\le i\le j-2\).
  \item \label{q4} The restriction map~\(\HT^{*}(X)\to H_{T'}^{*}(X)\)
    is surjective for all subtori~\(T'\) of~\(T\) of rank~\(r-j\).
  \item \label{q3} \(\HT^{*}(X)\) is free over all subrings~%
    \(H^{*}(BT'')\subset H^{*}(BT)=R\), where \(T''\) is a quotient of~\(T\) of rank~\(j\).
  \item \label{q2} \(\HT^{*}(X)\) is a \(j\)-th syzygy.
  \end{enumerate}
\end{theorem}

Recall that the entire Atiyah--Bredon sequence is exact
if condition~\eqref{q1} holds for \(j=r\) (Lemma~\ref{thm:exact-2}), and that
an 
\(R\)-module is an \(r\)-th syzygy if and only if it free over~\(R\).
Theorem~\ref{thm:conditions-partial-exactness} therefore
contains the Atiyah--Bredon result~\cite[Main Lemma]{Bredon:1974}
and its converse as special cases.

\begin{remark}
We know from Corollary~\ref{thm:HT-dual-to-hHT}
that the Atiyah--Bredon sequence
is exact at the first two positions 
if and only if the canonical map
\( 
  \HT^{*}(X) \to \Hom_{R}(\hHT_{*}(X),R)
\) 
is an isomorphism. By condition~\eqref{q2} above this is equivalent
to \(\HT^{*}(X)\) being a second syzygy, \ie, a reflexive \(R\)-module.
If this holds, then \(\HT^{*}(X)\) is a \(j\)-th syzygy for~\(j\ge3\) if and only
if
\begin{equation}
    H^{i}(\AB^{*}(X)) = \Ext_{R}^{i}(\hHT_{*}(X),R) = 0
\end{equation}
for all~\(1\le i\le j-2\). Note that this is not simply a reformulation
of condition~\eqref{tf3}
in Proposition~\ref{thm:torsionfree} because
\(\hHT_{*}(X)\) may not be the dual of~\(\HT^{*}(X)\)
even if this is true the other way around.
See Remark~\ref{rem:hHT-not-dual-to-HT} for an example.

Also observe that
the implication~\(\hbox{\eqref{q3}}\Rightarrow\hbox{\eqref{q2}}\) is not a purely algebraic fact:
For \(j=1\)  and \(f\in R\) not a product of linear polynomials,
the \(R\)-module~\(R/(f)\) is free over~\(\kk[a]\) for all \(0\ne a\in H^{2}(BT)\),
but it obviously has \(R\)-torsion.
\end{remark}

\begin{proof}
  For~\(j=0\) all conditions are true.
  
  \(\hbox{\eqref{q1}}\Rightarrow\hbox{\eqref{q2}}\):
  If \(j=1\), then the sequence~\(0\to\HT^{*}(X)\to\HT^{*}(X_{0})\) is exact,
  which means that \(\HT^{*}(X)\) is a first syzygy since \(\HT^{*}(X_0)\) is a free \(R\)-module.
  For~\(j\ge2\), consider a finitely generated free resolution
  \begin{equation*}
    F_{j-1} \to\dots\to F_0 \to \hHT_{*}(X) \to 0
  \end{equation*}
  Our assumption implies
  \(\Hom_R(\hHT_{*}(X),R) = \HT^{*}(X)\) and
  \(\Ext_R^i(\hHT_{*}(X),R)=0\) for~\(1\le i\le j-2\)
  by Theorem~\ref{thm:exthab}. Hence the sequence
  \begin{equation*}
    0 \to \HT^{*}(X) \to \Hom_R(F_0,R) \to\dots\to \Hom_R(F_{j-1},R)
  \end{equation*}
  is exact, exhibiting \(\HT^{*}(X)\) as a \(j\)-th syzygy.

  \(\hbox{\eqref{q2}}\Rightarrow\hbox{\eqref{q3}}\):
  If \(\HT^{*}(X)\) is finitely-generated over~\(H^{*}(BT'')\),
  this is a special case
  of the implication~\(\hbox{\eqref{tf1}}\Rightarrow\hbox{\eqref{tf2}}\)
  in Proposition~\ref{thm:torsionfree}.
  Alternatively, it follows from Hilbert's Syzygy Theorem,
  which shows that it is true for not finitely-generated modules as well.

  \(\hbox{\eqref{q3}}\Leftrightarrow\hbox{\eqref{q4}}\):
  Since \(\HT^{*}(X)=H_{T/T'}^{*}(X_{T'})\), this equivalence reduces to the statement
  that \(\HT^{*}(X)\) is free over~\(R\) if and only if the restriction map \(\HT^{*}(X)\to H^{*}(X)\)
  is surjective, \cf~Remark~\ref{rem:consequences-Serre}.
  
  \(\hbox{\eqref{q3}}\Rightarrow\hbox{\eqref{q1}}\):
  We do induction on~\(j\), for all~\(r\) and~\(X\) simultaneously. 
  By induction, we can assume \(H^{i}(\barAB^{*}(X))=0\) for all~\(i<j-2\).

  Choose a rational subspace~\(L\subset H^{2}(BT)\) of dimension~\(r-j+1\)
  transverse to \(M_{x}=\ker \bigl( H^{2}(BT)\to H^{2}(BT_x)\bigr)\) for all~\(x\in X_{j-1}\).
  (``Rational'' means that it has a basis lying in~\(H^{2}(BT;\Q)\).
  Such an~\(L\) exists because only finitely many tori~\(T_{x}^{0}\) occur,
  and \(M_{x}\) has dimension~\(\leq j-1\) for all~\(x\in X_{j-1}\).)
  Let \(S\subset R\) be the multiplicative set generated by~\(\tilde{L}=\bigl(L\cap H^{2}(BT;\Q)\bigr)\setminus\{0\}\).
  Since \(\dim M_{x}\ge j\)
  for all~\(x\notin X_{j-1}\),
  Lemma~\ref{thm:localization-orbit-filtration}\,\eqref{gg2} implies
  \(S^{-1}\HT^{*}(X_{i}, X_{i-1}) = 0\)  for~\(i\ge j\)
  and therefore \(S^{-1}H^{i}(\barAB^{*}(X))=0\) for all~\(i\notin\{j-2,j-1\}\).
  By Lemma~\ref{thm:exact-2}, this forces \(S^{-1}H^{j-2}(\barAB^{*}(X))=0\), too.
  Thus, it suffices to show that the localization map
  \begin{equation*}
    H^{j-2}(\barAB^{*}(X)) \to S^{-1} H^{j-2}(\barAB^{*}(X))
  \end{equation*}
  is injective, \ie, that no element~\(a\in S\) is a zero-divisor in~\(H^{j-2}(\barAB^{*}(X))\).

  We may assume \(a\in\tilde{L}\).
  Because \(a\) is rational, there is a subtorus~\(K\subset T\) such that \(H^{*}(BK)=R/a\).
  Since \(K\) is of rank~\(r-1\), the filtration degree of any point in~\(X\) decreases by at most~\(1\) as one compares the
  actions of \(T\)~and~\(K\). By our choice of~\(a\) as being transverse to~\(M_{x}\) for all~\(x\in X_{j-1}\),
  the orbit filtrations for~\(T\) and \(K\) coincide up to orbit dimension~\(j-2\).
  By assumption, \(\HT^{*}(X)\) is free over~\(\kk [a]\), hence \(H_{K}^{*}(X)=\HT^{*}(X)/a\).
  By Lemma~\ref{thm:localization-orbit-filtration}\,\eqref{gg1}
  the same holds for all pairs~\((X_{i},X_{i-1})\), \(0\le i\le j-2\), instead of~\(X\).
  Hence,  for \(-1\le i\le j-2\) the \(i\)-th term of the Atiyah--Bredon sequence 
  of~\(X\) with respect to~\(K\) is
  \(\barAB_{K}^{i}(X)=\barAB^{i}(X)\otimes_{\kk [a]}\kk \).
  Note that \(H_{K}^{*}(X)\) is free over all subrings \(H^{*}(BK'')\subset H^{*}(BK)\),
  where \(K''\) is a quotient of~\(K\) of rank~\(j-1\).
  We therefore have \(H^{i}(\barAB_{K}^{*}(X))=0\) for all~\(i<j-2\) by induction.
  Hence the middle term vanishes in the short exact sequence
  \begin{multline}
    0
    \longrightarrow
    H^{j-3}(\barAB^{*}(X))\otimes_{\kk[a]}\kk
    \longrightarrow
    H^{j-3}(\barAB_{K}^{*}(X)) \\
    \longrightarrow
    \Tor^{\kk [a]}_{1}(H^{j-2}(\barAB^{*}(X)),\kk )
    \longrightarrow
    0
  \end{multline}
  coming from the universal coefficient theorem,
  and so must do the \(\Tor\)~term.
  This implies that \(a\) is not a zero-divisor in~\(H^{j-2}(\barAB^{*}(X))\).
\end{proof}

\subsection{Poincaré duality spaces}
\label{sec:applications-PD}

Let \(X\) be a PD~space of formal dimension~\(n\), and let
and \(o_{T}\) be its equivariant orientation.
As in Remark~\ref{rem:equiv-PD},
we consider the equivariant Poincaré pairing
\begin{equation}
  \label{eq:PD-pairing}
  \HT^{*}(X)\times\HT^{*}(X)\to R,
  \quad
  (\alpha,\beta) \mapsto
  \pair{\alpha\cup\beta,o_{T}}
\end{equation}
which is (graded) symmetric and of degree~\(-n\). Recall that it 
is \emph{non-degenerate} if the induced morphism of \(R\)-modules
\begin{equation}
  \label{eq:HT-to-HomHTR}
  \HT^{*}(X) \to \Hom_{R}(\HT^{*}(X),R)
\end{equation}
is injective, and \emph{perfect} if \eqref{eq:HT-to-HomHTR}
is an isomorphism (of degree~\(-n\)). Note that non-degeneracy is equivalent
to the perfection of the localized pairing
\begin{equation}
  S^{-1}\HT^{*}(X)\times S^{-1}\HT^{*}(X)\to S^{-1}R,
\end{equation}
where \(S=R\setminus\{0\}\).

Our first observation could alternatively be deduced
from~\cite[Thm.~5.2.5]{AlldayPuppe:1993} or, in the smooth case,
from~\cite[Prop.~C.67]{GuilleminGinzburgKarshon:2002}.

\begin{proposition}
 The equivariant Poincaré pairing~\eqref{eq:PD-pairing} is non-degenerate
 if and only if \(\HT^{*}(X)\) is torsion-free.
\end{proposition}

\begin{proof}
  By Theorem~\ref{thm:exthab} and Poincaré duality, the map~\eqref{eq:HT-to-HomHTR}
  and the restriction map~\(\HT^{*}(X)\to\HT^{*}(X_{0})\) have the
  same kernel. By Theorem~\ref{thm:conditions-partial-exactness}
  (or the localization theorem),
  the latter map is injective if and only if \(\HT^{*}(X)\) is torsion-free.
\end{proof}

\begin{proposition}
  The following conditions are equivalent:
  \begin{enumerate}
  \item \label{p1} The Chang--Skjelbred sequence~\eqref{eq:chang-skjelbred} is exact.
  \item \label{p3} The  \(R\)-module~\(\HT^{*}(X)\) is reflexive.
  \item \label{p2} The equivariant Poincaré pairing~\eqref{eq:PD-pairing} is perfect.
  \end{enumerate}
  They are also equivalent to the conditions in Theorem~\ref{thm:conditions-partial-exactness} for~\(j=2\).
\end{proposition}

\begin{proof}
  \(\hbox{\eqref{p1}}\Leftrightarrow\hbox{\eqref{p3}}\) is a rephrasing
  of Theorem~\ref{thm:conditions-partial-exactness} because
  reflexive \(R\)-modules are exactly the second syzygies.
  \(\hbox{\eqref{p1}}\Leftrightarrow\hbox{\eqref{p2}}\)
  follows from Theorem~\ref{thm:exthab} and equivariant Poincaré duality.
\end{proof}

\begin{remark}
  \label{rem:ggk}
  Similarly, any \(R\)-linear map~\(\HT^{*}(X)\to R\) is the Poincaré pairing with some class~\(\alpha\in\HT^{*}(X)\)
  if and only if the Chang--Skjelbred (or Atiyah--Bredon) sequence is exact at~\(\HT^{*}(X_{0})\).
  The example given in Section~\ref{sec:mutants} shows that this is not always the case.
  This clarifies a point raised by Guillemin--Ginzburg--Karshon~\cite[App.~C.8.2]{GuilleminGinzburgKarshon:2002}.
\end{remark}

\medbreak

In~\cite[Prop.]{Allday:1985}, Allday proved the following
for a PD~space~\(X\):
If \(\HT^{*}(X)\) has homological dimension \(1\),
then it has \(R\)-torsion.
In particular, if~\(r=2\), \ie, if \(T=S^{1}\times S^{1}\),
then \(\HT^{*}(X)\) is torsion-free if and only if it is free.
In other words, if \(r=2\), then the Atiyah--Bredon sequence for~\(X\)
is exact if and only if it is exact at~\(\HT^{*}(X)\).
This equivalence breaks down for~\(r>2\);
see Section~\ref{sec:mutants} for a counterexample.
The correct generalization of Allday's result is as follows.

\goodbreak

\begin{proposition} \label{thm:generalization-allday} \( \) 
  \begin{enumerate} 
  \item \label{thm:torsion-free-hd-1}
    If \(\HT^{*}(X)\) is a \(j\)-th syzygy and of homological dimension at most~\(j\),
    then it is free over~\(R\).
  \item \label{thm:exactness-PD}
    If \(\HT^{*}(X)\) is a \(j\)-th syzygy for some~\(j\ge r/2\), then it is free over~\(R\). Equiv\-a\-lently,
    if the Atiyah--Bredon sequence
    for~\(X\) is exact at all positions \(i<r/2-1\),
    then \(\HT^{*}(X)\) is free over~\(R\).
  \end{enumerate}
\end{proposition}

\begin{proof}
  If \(\HT^{*}(X)=\hHT_{*}(X)[n]\) is of homological dimension~\(\le j\), then \(H^{i}(\AB^{*}(X))=\Ext_{R}^{i}(\hHT_{*}(X),R)\)
  vanishes for~\(i>j\). On the other hand, if it is a \(j\)-th syzygy, then \(H^{i}(\AB^{*}(X))=0\) for~\(i\le j-2\)
  by Theorem~\ref{thm:conditions-partial-exactness}. Lemma~\ref{thm:exact-2} now proves the first claim.

  The two hypotheses in the second claim are equivalent by Theorem~\ref{thm:conditions-partial-exactness}.
  They imply that \(\HT^{*}(X)\)
  admits a regular sequence of length~\(j\),
  so that its homological dimension is bounded by~\(r-j\).
  Now use the first part with~\(r-j\le j\).
\end{proof}


\section{Examples}
\label{sec:examples}

\subsection{Non-compact examples}
\label{sec:toric}

In this final section we apply
our results, in particular the criteria
for partial exactness of the Atiyah--Bredon sequence
given in
Theorem~\ref{thm:conditions-partial-exactness}
and Proposition~\ref{thm:generalization-allday}\,\eqref{thm:exactness-PD},
to several orientable smooth manifolds.
By Theorem~\ref{thm:conditions-partial-exactness},
partial exactness is related to syzygies.
The syzygies in our examples will be the ``Koszul syzygies''
discussed in Section~\ref{sec:koszul-resolution}.

Let \(X\) be a \(T\)-space such that \(\HT^{*}(X)\) is not free over~\(R\).
By Theorem~\ref{thm:conditions-partial-exactness} 
and Lemma~\ref{thm:exact-2},
this means that the Atiyah--Bredon sequence must be
non-exact at two non-adjacent positions. In this section
we present, for any~\(r\ge1\), a ``minimally non-exact'' example in the sense that
\begin{equation}
  \label{eq:minimal-non-exact}
  H^{i}(\barAB^{\,*}(X)) \cong
  \begin{cases}
    \kk  & \text{if \(i=r-2\),} \\
    \kk [-1] & \text{if \(i=r\),} \\
    0 & \text{otherwise,}
  \end{cases}
\end{equation}
where \(\barAB^{\,*}(X)\) denotes the (augmented) Atiyah--Bredon sequence~\eqref{eq:atiyah-bredon}.
In particular, \(\HT^{*}(X)\) will necessarily be
an \((r-1)\)-st syzygy. 

Our example is
\begin{equation}
  X=(\CP^{1})^{r}\setminus\{(N,\dots,N),(S,\dots,S)\},
\end{equation}
where \(N\)~and~\(S\) are
the two fixed points for the standard action of~\(S^{1}\) on~\(\CP^{1}=S^{2}\),
and \(T=(S^{1})^{r}\) acts on~\(X\) by the restriction of the componentwise action.

Both \(X\)~and~\(Y=(\CP^{1})^{r}\) are smooth toric varieties,
hence their equivariant cohomology can be described as
Stanley--Reisner rings, see~\cite[Thm.~8]{BifetDeConciniProcesi:1990}
or~\cite[Prop.~1.3 \&~2.2]{Brion:1996}.
The \(\kk\)-algebra~\(\HT^{*}(Y)\) is generated
by degree~\(2\) elements \(u_{i}\)~and~\(v_{i}\) subject to the
relations~\(u_{i}v_{i}=0\) for~\(1\le i\le r\). It is an \(R\)-module
via the map of algebras~\(R\to\HT^{*}(Y)\) sending \(t_{i}\) to~\(u_{i}-v_{i}\).
Moreover, \(\HT^{*}(X)\) is isomorphic to the quotient of~\(\HT^{*}(Y)\)
by the ideal (or, equivalently, submodule) generated by
\begin{equation}
  U = u_{1}\cdots u_{r}
  \qquad\text{and}\qquad
  V = v_{1}\cdots v_{r} = \prod_{i=1}^{r}(u_{i}-t_{i}).
\end{equation}

In what follows, isomorphisms refer to isomorphisms of \(R\)-modules; product structures are not considered.
We will use the isomorphism 
\begin{equation}
  \HT^{*}(Y) = \bigoplus_{I\subset[r]} R\,u_{I},
\end{equation}
where \([r]=\{1,\dots,r\}\) and \(u_{I}\) is the product of the~\(u_{i}\) with~\(i\in I\).
Let \(N\) be the submodule spanned by the \(u_{I}\) with~\(|I|\le r-2\),
and let \(N'\) be its isomorphic image in~\(\HT^{*}(Y)/(U,V)=\HT^{*}(X)\).
Then \(\HT^{*}(X)/N'\) is isomorphic to the quotient of
\begin{equation}
  \label{iso-U-N}
  \HT^{*}(Y)\bigm/\bigl((U)+N\bigr) \cong \bigoplus_{i=1}^{r} R\,u_{[r]\setminus\{i\}}
\end{equation}
by the image of~\(V\). Under the isomorphism~\eqref{iso-U-N},
\(V\) corresponds to the element
\begin{equation}
  -\sum_{i=1}^{r} t_{i}\,u_{[r]\setminus\{i\}},
\end{equation}
which also generates the image of the differential~\(\delta_{r}\)
in the Koszul resolution~\eqref{eq:koszul-resolution}.
We therefore have established a short exact sequence
\begin{equation}
  \label{eq:short-exact-HTX}
  0 \longrightarrow
  \bigoplus_{i=0}^{r-2} R[2i]^{\binom{r}{i}} \longrightarrow
  \HT^{*}(X) \longrightarrow
  K_{r-1}[2(r-1)] \longrightarrow 0.
\end{equation}
It splits by Lemma~\ref{thm:Ext-Koszul}, and we obtain
\begin{equation}
  \label{eq:toric-example-HT}
  \HT^{*}(X) \cong \bigoplus_{i=0}^{r-2} R[2i]^{\binom{r}{i}} \oplus K_{r-1}[2(r-1)].
\end{equation}
In particular, \(\HT^{*}(X)\) is an \((r-1)\)-st syzygy,
and the Atiyah--Bredon sequence for~\(X\) is exact at all positions~\(i\le r-3\).

We compute the equivariant homology of~\(X\) via the universal coefficient spectral sequence
and Lemma~\ref{thm:Ext-Koszul}. We get 
\begin{equation}
  \Ext_{R}^{j}(\HT^{*}(X),R) \cong
  \begin{cases}
    \bigoplus_{i=0}^{r-2} R[-2i]^{\binom{r}{i}} \oplus K_{2}[-2(r-2)] & \text{if \(j = 0\),} \\
    \kk[-2r]  & \text{if \(j = 1\),} \\
    0 & \text{otherwise.}
  \end{cases}
\end{equation}
Hence, no extension problem arises and
\begin{equation}
  \label{eq:toric-example-hHT}
  \hHT_{*}(X) \cong \bigoplus_{i=0}^{r-2} R[-2i]^{\binom{r}{i}} \oplus K_{2}[-2(r-2)] \oplus \kk [1-2r].
\end{equation}
To confirm \eqref{eq:minimal-non-exact} for the remaining cases~\(r-2\le j\le r\),
we invoke Lemma~\ref{thm:Ext-Koszul} again and obtain for~\(r\ge3\)
\begin{equation}
  \Ext_{R}^{j}(\hHT_{*}(X),R) \cong
  \begin{cases}
    \bigoplus_{i=0}^{r-2} R[2i]^{\binom{r}{i}} \oplus K_{r-1}[2(r-1)] & \text{if \(j = 0\),} \\
    \kk  & \text{if \(j = r-2\),} \\
    \kk [-1] & \text{if \(j = r\),} \\
    0 & \text{otherwise.}
  \end{cases}  
\end{equation}
For~\(r=2\) one has
\begin{equation}
  \Ext_{R}^{j}(\hHT_{*}(X),R) \cong
  \begin{cases}
    R^{2} & \text{if \(j = 0\),} \\
    0  & \text{if \(j = 1\),} \\
    \kk [-1] & \text{if \(j = 2\).} \\
  \end{cases}  
\end{equation}
The case~\(i=0\) of~\eqref{eq:minimal-non-exact} follows here from Lemma~\ref{thm:exact-2}.
In the case~\(r=1\) one of course finds
\begin{equation}
  \Ext_{R}^{j}(\hHT_{*}(X),R) \cong
  \begin{cases}
    0  & \text{if \(j = 0\),} \\
    \kk [-1] & \text{if \(j = 1\).} \\
  \end{cases}  
\end{equation}

\smallskip

This example also illustrates the following point:

\begin{remark}
  \label{rem:hHT-not-dual-to-HT}
  Assume that \(\HT^{*}(X)\) is not free.
  Then the equivariant homology~\(\hHT_{*}(X)\) may well have \(R\)-torsion,
  even if \(\HT^{*}(X)\) is ``as close to being free as possible'',
  that is, an \((r-1)\)-st syzygy.
  In particular, that \(\HT^{*}(X)\) is reflexive
  does not imply that \(\hHT_{*}(X)\) is so,
  nor that it is the \(R\)-dual of~\(\HT^{*}(X)\).
\end{remark}

\begin{remark}
  Instead of two fixed points
  one could remove small \(T\)-stable open neighbourhoods of them from~\((S^{2})^{r}\).
  This way one would obtain a smooth manifold~\(Y\) equivariantly homotopy-equivalent to~\(X\)
  which is compact and with boundary instead of non-compact without boundary.
  It follows from Proposition~\ref{thm:generalization-allday}\,\eqref{thm:exactness-PD}
  that there is no Poincaré duality space satisfying
  \eqref{eq:minimal-non-exact} for~\(r\ge2\).
\end{remark}

\subsection{The mutant}
\label{sec:mutants}

Let \(X\) be the \(7\)-dimensional ``mutant'' constructed in~\cite{FranzPuppe:2008}.
This is a compact orientable manifold with a smooth action of the torus~\(T=(S^{1})^{3}\),
homeomorphic to the connected sum of \(3\)~copies of~\(S^{3}\times S^{4}\).
As shown in~\cite{FranzPuppe:2008},
the equivariant cohomology of~\(X\) is the \(R\)-module
\begin{align}
  \label{eq:mutant-HT}
  \HT^{*}(X) &\cong 
  R\oplus\m[1]\oplus R[6]\oplus R[7], \\
\intertext{which is torsion-free, but not free.
By Poincaré duality, one gets}
  \hHT_{*}(X) &\cong R\oplus R[-1]\oplus\m[-6]\oplus R[-7],
\end{align}
hence
\begin{equation}
  \label{eq:mutant-hHT}
  \Ext_{R}^{j}(\hHT_{*}(X),R) \cong
  \begin{cases}
    R\oplus R[1]\oplus R[6]\oplus R[7] & \text{if \(j = 0\),} \\
    \kk  & \text{if \(j = 2\),} \\
    0 & \text{otherwise.}
  \end{cases}
\end{equation}
As \(\HT^{*}(X)\) is not free over~\(R\), the Atiyah--Bredon sequence cannot be exact.
In fact, one finds
that its cohomology is
\begin{equation}
  H^{i}(\barAB^{*}(X)) \cong
  \begin{cases}
    \kk [1] & \text{if \(i = 0\),} \\
    \kk  & \text{if \(i = 2\),} \\
    0 & \text{otherwise,}
  \end{cases}
\end{equation}
which matches \eqref{eq:mutant-hHT}~and~\eqref{eq:mutant-HT}.


\begin{thebibliography}{99}
  
\bibitem{Allday:1985}
  C.~Allday,
  \newblock A family of unusual torus group actions,
  \newblock pp.~107--111 in:
  \newblock R.~Schultz (ed.),
  \newblock Group actions on manifolds (Boulder, CO, 1983),
  \newblock \textit{Contemp.\ Math.}~\textbf{36} (1985)

\bibitem{AlldayPuppe:1993}
C.~Allday, V.~Puppe,
\newblock \textit{Cohomological methods in transformation groups},
  Cambridge Stud.\ Adv.\ Math.~\textbf{32},
\newblock Cambridge Univ.\ Press, Cambridge 1993

\bibitem{Atiyah:1974}
M.~F.~Atiyah,
\newblock \textit{Elliptic operators and compact groups},
\newblock LNM~\textbf{401}, Springer, Berlin 1974;
\newblock \doi{10.1007/BFb0057821}

\bibitem{BarthelBrasseletFieselerKaup:2002}
G.~Barthel, J.-P.~Brasselet, K.-H.~Fieseler, L.~Kaup,
\newblock Combinatorial intersection cohomology for fans,
\newblock \textit{Tohoku Math.\ J.\ (2)}~\textbf{54} (2002), 1--41;
\newblock \doi{10.2748/tmj/1113247177}

\bibitem{BifetDeConciniProcesi:1990}
E.~Bifet, C.~De\,Concini, C.~Procesi,
\newblock Cohomology of regular embeddings,
\newblock \textit{Adv.\ Math.}~\textbf{82} (1990), 1--34;
\newblock \doi{10.1016/0001-8708(90)90082-X}



\bibitem{Borel:1957}
A.~Borel,
\newblock The Poincaré duality in generalized manifolds,
\newblock \textit{Michigan Math.\ J.}~\textbf{4} (1957), 227--239;
\newblock \doi{10.1307/mmj/1028997954}


\bibitem{Bredon:1974}
G.~E.~Bredon,
\newblock The free part of a torus action and related numerical equalities,
\newblock \textit{Duke Math.\ J.}~\textbf{41} (1974), 843--854;
\newblock \doi{10.1215/S0012-7094-74-04184-2}


\bibitem{Brion:1996}
M.~Brion,
\newblock Piecewise polynomial functions, convex polytopes and enumerative geometry,
\newblock pp.~25--44 in: P.~Pragacz (ed.), \textit{Parameter spaces},
  Banach Cent.\ Publ.\ \textbf{36}, Warszawa 1996
  

\bibitem{BrunsHerzog:1998}
W.~Bruns, J.~Herzog,
\newblock \textit{Cohen--Macaulay rings},
\newblock 2nd~ed., Cambridge Stud.\ Adv.\ Math.~\textbf{39},
\newblock Cambridge Univ.~Press, Cambridge 1998

\bibitem{BrunsVetter:1988}
W.~Bruns, U.~Vetter,
\newblock \textit{Determinantal rings},
\newblock LNM~\textbf{1327}, Springer, Berlin 1988;
\newblock \doi{10.1007/BFb0080378}

\bibitem{Brylinski:1992}
  J.-L.~Brylinski,
  \newblock Equivariant intersection cohomology,
  \newblock pp.~5--32 in:
  \newblock V.~Deodhar (ed.),
  \newblock Kazhdan--Lusztig theory and related topics (Chicago, IL, 1989),
  \newblock \textit{Contemp.\ Math.}~\textbf{139}, AMS, Providence, RI, 1992

\bibitem{ChangSkjelbred:1974}
T.~Chang, T.~Skjelbred,
\newblock The topological Schur lemma and related results,
\newblock \textit{Ann.\ Math.~(2)}~\textbf{100} (1974), 307--321;
\newblock \doi{10.2307/1971074}

\bibitem{DeConciniProcesiVergne:2010}
C.~De\,Concini, C.~Procesi, M.~Vergne,
\newblock Vector partition functions and the index of transversally elliptic operators,
\newblock \textit{Transformation Groups}~\textbf{15} (2010), 775--811;
\newblock \href{http://dx.doi.org/10.1007/s00031-0100-9101-x}{10.1007/s00031-0100-9101\nobreakdash-x} 

\bibitem{DeConciniProcesiVergne:2011}
C.~De\,Concini, C.~Procesi, M.~Vergne,
\newblock Infinitesimal index: cohomology computations,
\newblock \textit{Transformation Groups}~\textbf{16} (2011), 717--735;
\newblock \doi{10.1007/s00031-011-9144-7}



\bibitem{Duflot:1983}
J.~Duflot,
\newblock Smooth toral actions,
\newblock \textit{Topology}~\textbf{22} (1983), 253--265;
\newblock \doi{10.1016/0040-9383(83)90012-5}

\bibitem{EdidinGraham:1998}
  D.~Edidin, W.~Graham,
  \newblock Equivariant intersection theory,
  \newblock \textit{Invent.\ Math.}~\textbf{131} (1998), 595--634;
  \newblock \doi{10.1007/s002220050214}

\bibitem{Eisenbud:2005}
D.~Eisenbud,
\newblock \textit{The geometry of syzygies},
\newblock GTM~{\bf 229}, Springer, New York 2005;
\newblock \doi{10.1007/b137572}

\bibitem{FelixHalperinThomas:1995}
Y.~Félix, S.~Halperin, J.-C.~Thomas,
\newblock Differential graded algebras in topology,
\newblock pp.~829--865 in:
\newblock I.\,M.~James (ed.), \textit{Handbook of algebraic topology},
\newblock North-Holland, Amsterdam 1995;
\newblock \doi{10.1016/B978-044481779-2/50017-1}


\bibitem{Franz:2003}
M.~Franz,
\newblock Koszul duality and equivariant cohomology for tori,
\newblock \textit{Int.\ Math.\ Res.\ Not.}~\textbf{42} (2003), 2255-2303;
\newblock \doi{10.1155/S1073792803206103}

\bibitem{FranzPuppe:2007}
M.~Franz, V.~Puppe,
\newblock Exact cohomology sequences with integral coefficients
  for torus actions,
\newblock \textit{Transformation Groups}~\textbf{12} (2007), 65--76;
\newblock \doi{10.1007/s00031-005-1127-0}

\bibitem{FranzPuppe:2008}
M.~Franz, V.~Puppe,
\newblock Freeness of equivariant cohomology and
    mutants of compactified representations,
\newblock pp.~87--98 in: M.~Harada \emph{et al.}~(eds.), Toric Topology (Osaka, 2006),
  \textit{Contemp.\ Math.}~\textbf{460}, AMS, Providence, RI, 2008

\bibitem{FranzPuppe:2011}
M.~Franz, V.~Puppe,
\newblock Exact sequences for equivariantly formal spaces,
\newblock \textit{C.\ R.\ Math.\ Acad.\ Sci.\ Soc.\ R.\ Can.}~\textbf{33} (2011), 1--10

\bibitem{GoertschesToeben:2010}
O.~Goertsches, D.~T\"oben,
\newblock Torus actions whose equivariant cohomology is Cohen--Macaulay,
\newblock \textit{J.\ Topology}~\textbf{3} (2010), 819--846;
\newblock \doi{10.1112/jtopol/jtq025}

\bibitem{GoreskyKottwitzMacPherson:1998}
M.~Goresky, R.~Kottwitz, R.~MacPherson,
\newblock Equivariant cohomology, Koszul duality, and the localization theorem,
\newblock \textit{Invent.\ Math.}~\textbf{131} (1998), 25--83;
\newblock \doi{10.1007/s002220050197}

\bibitem{GuilleminGinzburgKarshon:2002}
  V.~Guillemin, V.~Ginzburg, Y.~Karshon,
  \newblock Moment maps, cobordisms, and Hamiltonian group actions,
  \newblock AMS, Providence, RI, 2002

 

\bibitem{Jones:1987}
  J.~D.~S.~Jones,
  \newblock Cyclic homology and equivariant homology,
  \newblock \textit{Invent.\ Math.}~\textbf{87} (1987), 403--423;
  \newblock \doi{10.1007/BF01389424}



\bibitem{McCleary:2001}
J.~Mc{C}leary.
\newblock \textit{A user's guide to spectral sequences}, 2nd ed.,
  Cambridge Stud.\ Adv.\ Math.~\textbf{58},
\newblock Cambridge Univ.\ Press, Cambridge 2001
    

\bibitem{Schenck:1997}
H.~Schenck,
\newblock A spectral sequence for splines,
\newblock \textit{Adv.\ in Appl.\ Math.}~\textbf{19} (1997), 183-199;
\newblock \doi{10.1006/aama.1997.0534}

\bibitem{Schenck:2012}
H.~Schenck,
\newblock Equivariant Chow cohomology of nonsimplicial toric varieties,
\newblock \textit{Trans.\ Amer.\ Math.\ Soc.}~\textbf{364} (2012), 4041-4051;
\newblock \doi{10.1090/S0002-9947-2012-05409-2}

  



\end{thebibliography}
\end{document}